\documentclass[11pt,reqno]{amsart}
\usepackage[utf8]{inputenc}
\usepackage[T1]{fontenc}
\usepackage{hyperref}
\usepackage[all,cmtip]{xy}
\usepackage{amsmath}
\newtheorem{theorem}{Theorem}[section]
\newtheorem{cor}[theorem]{Corollary}
\newtheorem{lema}[theorem]{Lemma}
\newtheorem{prop}[theorem]{Proposition}
\newtheorem{example}[theorem]{Example}
\newtheorem{obs}[theorem]{Remark}
\newtheorem{defini}[theorem]{Definition}
\newtheorem{defprop}[theorem]{Proposition-Definition}

\newtheorem{conj}{Conjecture}

\newtheorem*{th*}{Question}

\newtheorem{thmx}{Theorem}
\theoremstyle{plain} % just in case the style had changed
\newcommand{\thistheoremname}{}
\newtheorem*{genericthm*}{\thistheoremname}
\newenvironment{namedthm*}[1]
  {\renewcommand{\thistheoremname}{#1}%
   \begin{genericthm*}}
  {\end{genericthm*}}

\numberwithin{equation}{section}

\newcommand{\C}{\mathbb{C}}
\newcommand{\N}{\mathbb{N}}

\newcommand{\R}{\mathbb{R}}

\newcommand{\A}{\mathbb{A}}   
\newcommand{\D}{\mathbb{D}}
\newcommand{\Int}{\mathrm{Int}}
\newcommand{\an}{\mathrm{an}}

\newcommand{\red}{\mathrm{red}}
\newcommand{\T}{\mathcal{T}}
\newcommand{\ev}{\mathrm{Ev}}
\newcommand{\mor}{\mathrm{Mor}}

\newcommand{\M}{\mathrm{M}}
\newcommand{\al}{\alpha}

\begin{document}

\title[Non-Archimedean normal families] % running head version
{Non-Archimedean normal families}

\author{Rita Rodríguez Vázquez}
\address[Rita Rodríguez Vázquez]{CMLS, École polytechnique, CNRS, Université Paris-Saclay, 91128 Palaiseau Cedex, France
}
%% Note the doubled @@:
\email{rita.rodriguez-vazquez@polytechnique.edu}

\date{December 22, 2017}

\begin{abstract}
We present several results on the compactness of the space of morphisms between  analytic spaces in the sense of Berkovich. 
We show that under certain conditions on the source, every sequence of analytic maps having an affinoid target has a subsequence that converges pointwise to a continuous map.
We also study the class of continuous maps that arise in this way.
 Locally, they turn analytic after a certain base change.
 Our results naturally lead to a definition of normal families.
We give some applications  to the dynamics of an endomorphism $f$ of the projective space. 
 We introduce two notions of Fatou set and generalize to the non-Archimedan  setting a theorem of Ueda stating that  
 every Fatou component is hyperbolically imbedded in the projective space.
\end{abstract}

\maketitle

\addtocontents{toc}{\protect\setcounter{tocdepth}{1}}
\tableofcontents

\section{Introduction}

The classical Montel's theorem asserts that any family of holomorphic functions on a domain in $\C^n$ with values in a ball
is relatively compact for the topology of the local uniform convergence \cite{Montel}.
The proof uses Cauchy's estimates to obtain a uniform bound on the derivatives. By Ascoli-Arzelà's theorem  the family then is equicontinuous and the result follows.

 This result has several applications in complex dynamics.
It also plays an important role in the study of Kobayashi hyperbolic complex analytic spaces.
For instance, it is closely related to Zalcman's reparametrization lemma \cite{Zalcman}, which is a key ingredient  in the proof of Brody's Lemma \cite{Brody}, characterizing compact  Kobayashi hyperbolic complex analytic spaces in terms of the non-existence of entire curves.
\medskip

The aim of this paper is to study 
the compactness properties of the space of morphisms between analytic spaces defined over a non-Archimedean complete field, in analogy to the classical Montel's Theorem.
We therefore fix a non-Archimedean complete valued field  $k$ that is nontrivially valued.

\medskip

An approach to this problem using equicontinuity has already been treated in the literature.
Hsia gave in  \cite{Hs} an equicontinuity criterion  for families of meromorphic functions on  a disk.
In \cite{SilverKawagreen}, the  Fatou set  of a morphism of the projective space is defined as the equicontinuity locus of the family of iterates with respect to the chordal metric.
 However, this approach is limited by the fact that one cannot apply Ascoli-Arzelà's theorem in this context. 

\medskip

We will work on  analytic spaces as defined in  \cite{Berk,Berk2}.
 The main reason is that analytic spaces in the sense of Berkovich have good topological properties: they are locally compact and locally pathwise connected, what makes them a more adapted framework to arguments of analytic nature. 
The analytic spaces we shall be mostly interested in  are Berkovich analytifications  of  projective varieties. Recall that the set of closed points of such a variety forms a dense subset  of its analytification with empty interior if $k$ is not trivially valued.
We shall refer to these points  as \emph{rigid points}. 
The previously mentioned equicontinuity results concern only the set of rigid points.

\medskip

More recently, Favre, Kiwi and Trucco proved  an analogue of Montel's theorem  on the Berkovich analytic projective line $\mathbb{P}^{1, \an}_k$, see \cite{FKT}.
They show that when $k$ is algebraically closed and has residue  characteristic 0, then every sequence of analytic maps from \emph{any} open connected subset $X$ of $\mathbb{P}^{1, \an}_k$  avoiding three points  has a subsequence that is pointwise converging to a continuous map $X \to \mathbb{P}^{1, \an}_k$. 
They made extensive use of Berkovich's geometry 
and their strategy benefits from the tree structure of $\mathbb{P}^{1, \an}_k$.
 
\medskip

We explore the higher dimensional case, and consequently use deeper facts from Berkovich theory.
Of particular relevance for us is the theorem by Poineau stating that   compact analytic spaces are sequentially compact, see \cite{Poineau}.
This result is non-trivial, since Berkovich spaces are not metrizable in general.
 We show:

\begin{thmx}\label{thm a}
Let $k$ be a non-Archimedean complete field that is nontrivially valued.
Let $X$ be a good, reduced, $\sigma$-compact, boundaryless $k$-analytic space.  Let  $Y$ be a strictly  $k$-affinoid space. 

Then, every sequence of analytic maps $f_n: X \to Y$ admits a pointwise converging subsequence whose limit is continuous.
\end{thmx}

The seemingly complicated hypothesis on the source space $X$ are not such in fact.
We refer the reader to \S \ref{section intro berko} for a detailed discussion on the technical assumptions on $X$.
For the moment, let us indicate that two important classes of $k$-analytic spaces satisfy these properties: analytifications of algebraic varieties and connected components of the analytic interior of an affinoid space. 
The latter will be referred to as  \emph{basic tubes}.
They have been thoroughly studied by Bosch and Poineau, see \cite{Bosch, Poineaucomposantes}. 

\smallskip

Remark that the boundaryless assumption is crucial, as
 problems arise even in the affinoid case.
 Indeed, as pointed out  in \cite[\S 4.2]{FKT}, consider for instance the sequence of analytic maps from the closed unit disk $\bar{\D}$ into itself $f_n: z\mapsto z^{2^{n!}}$. For every $n \in \N$, the Gauss point is a fixed point for $f_n$.
One can show that $f_n$ is pointwise converging, but its limit map is zero on the whole open unit disk and hence not continuous.

\medskip

 In view of Theorem \ref{thm a}, we say that a family of analytic maps $\mathcal{F}$ from a boundaryless $k$-analytic space $X$  into  a compact  space $Y$ is normal at a point point $x \in X$ if there exists a neighbourhood $V \ni x$ such that every sequence $\{ f_n \}$ in $\mathcal{F}$ admits a subsequence $f_{n_j}$ that is pointwise converging on $V$ to some continuous map $f: V \to Y$.

\bigskip

We now turn to the problem of describing the limits of pointwise converging analytic maps. 
As opposed to the complex setting,  one cannot expect the limit maps from Theorem \ref{thm a} to be analytic. 
Indeed, when $k$ is algebraically closed and non-trivially valued, any constant map $f:X \to Y$, $f\equiv y \in Y$, can be realized as the limit of constant analytic maps.
However, $f$ is analytic if and only if $y$ is rigid.
\medskip

Inspite of not being analytic in general, the continuous limit maps obtained in Theorem \ref{thm a} are of a very particular kind:
 they turn analytic  after a suitable base change.
In order to specify this phenomenon precisely,  we rely again in a crucial way on the results of Poineau.
Let $X$ be  a $k$-analytic space.
For every complete extension $K$ of $k$, we denote by $\pi_{K/k}: X_K \to X$ the usual base change morphism. 
Every $k$-point in $X$ defines a $K$-point in $X_K$ in a natural manner. 
When the base field $k$ is algebraically closed, Poineau \cite{Poineau} shows that this inclusion admits a unique \emph{continuous} extension 
 $\sigma_{K/k}: X \to X_K$, which by construction defines a section of $\pi_{K/k}$.

\begin{thmx}\label{thm b}
Let $k$ be a non-Archimedean algebraically closed  complete   field that is nontrivially valued and $X$ a good, reduced, boundaryless strictly $k$-analytic space.  Let  $Y$ be a $k$-affinoid space. Let $f_n: X \to Y$ be a sequence of analytic maps converging pointwise to a continuous map $f$. 

Then, for any point $x\in X$ one can find an affinoid neighbourhood $Z$ of $x$, a complete extension $K/k$ and a $K$-analytic map $F: Z_K \to Y_K$ such that 
\begin{equation*}
f|_Z  = \pi_{K/k} \circ F \circ \sigma_{K/k}.
\end{equation*}
\end{thmx}

It would be interesting to find a $K$-analytic map $F$ such that the stronger condition 
 $\pi_{K/k} \circ F  = f \circ \pi_{K/k}$ holds, but our proof does not show this.

\medskip

Let us explain the proof of Theorem \ref{thm b} in the case where $X$ is the open $r$-dimensional polydisk $\D^r$ and $Y$ the closed $s$-dimensional polydisk $\bar{\D}^s$. 
The key idea is to view the set of all analytic maps from  $\D^r$ to $\bar{\D}^s$ as the set of rigid points of an infinite dimensional polydisk $\mor (\D^r, \bar{\D}^s)$.
This procedure can be easily illustrated in the polynomial case.
Observe that a polynomial map sending $\D^r$ into $\bar{\D}^s$ is given by finitely many coefficients in the base field $k$ with norm at most 1, and so  defines a rigid point in an appropriate dimensional closed unit polydisk. This procedure can be done similarly for general analytic maps. 
In this case, the coefficients define a rigid point in an infinite dimensional polydisk denoted $\mor (\D^r, \bar{\D}^s)$.

Now take a sequence $f_n: \D^r \to \bar{\D}^s$ of analytic maps, associated to a sequence of rigid points $\{ \al_n \}$ in $\mor (\D^r, \bar{\D}^s)$. 
It can be showed that the fact that $f_n$ converges pointwise to some continuous map $f$
 amounts for $\al_n$ to converging to some point  $\al$ in $\mor (\D^r, \bar{\D}^s)$. Observe that $\al$ is not rigid in general, but 
  after a base change by  $\mathcal{H}(\al)$,  the complete residue field at $\al$, the point $\al$ can be lifted to a rigid point in $\mor (\D^r, \bar{\D}^s)_{\mathcal{H}(\al)}$.
  This point defines a $\mathcal{H}(\al)$-analytic map $F: \D^r_{\mathcal{H}(\al)} \to \bar{\D}^s_{\mathcal{H}(\al)}$ that satisfies the  equality $f= \pi_{\mathcal{H}(\al)/k} \circ F \circ \sigma_{\mathcal{H}(\al)/k}$. 
Observe that $F$  is not canonical, as it depends on the choice of the rigid point in $\mor (\D^r, \bar{\D}^s)_{\mathcal{H}(\al)}$ lying over $\al$.

\smallskip

We go beyond Theorem \ref{thm b} and
 show that to any point $\alpha$ in $\mor (\D^r, \bar{\D}^s)$ one can associate a
   continuous map from  $\D^r$ to $\bar{\D}^s$   in a continuous way,
in the sense that for any sequence of points $\al_n$ in $\mor (\D^r, \bar{\D}^s)$ converging to $\al \in \mor (\D^r, \bar{\D}^s)$, 
 the corresponding sequence of continuous maps converges everywhere pointwise to the continuous map associated to $\al$. 
In Section \ref{section mor}  we detail this correspondence.

\medskip

This result suggests the following definition.
 We say that a continuous map $f$  between analytic spaces
is \emph{weakly analytic} if it is locally of the form $f  = \pi_{K/k} \circ F \circ \sigma_{K/k}$ for some complete extension $K$ of $k$ and some $K$-analytic map $F$.
In fact, weakly analytic maps can be characterized as being locally the pointwise limit of analytic maps.  
In \S \ref{section weakly analytic} we shall prove that  weakly analytic maps share many properties  with analytic maps, such as  an isolated zero principle on curves.

\bigskip

We give  applications of Theorem \ref{thm a}  to the dynamics of  an  endomorphism $f$ of the $k$-analytic projective space $\mathbb{P}^{N, \an}_k$  of degree at least 2.
Kawaguchi and Silverman associated a   non-Archimedean Green function $G_f$ to $f$ in \cite{Silvkawadynamics,SilverKawagreen}, 
generalizing the  classical complex construction by Hubbard \cite{HubbardHenon} and Fornaess and Sibony \cite{FS2}.
We attach to $f$ two different notions of Fatou sets.
We  define the \emph{normal Fatou set} $F_{\mathrm{norm}} (f)$ of $f$   as the normality locus of the family of the iterates $\{f^n\}$.
Next, we define the \emph{harmonic Fatou set} $F_{\mathrm{harm}} (f)$ as the set where
 the non-Archimedean Green function $G_f$ of $f$ introduced by Kawaguchi-Silverman is strongly pluriharmonic in the sense of \cite{CLheights}.

\smallskip

In Proposition \ref{prop fatou pluriharm} we show
that the harmonic Fatou set of $f$ can be characterized in terms 
of a sort of equicontinuity property for the iterates of $f$.
 Its proof follows its complex  counterpart.
It is now a consequence of Theorem \ref{thm a} that $F_{\mathrm{harm}} (f)$ is contained in $F_{\mathrm{norm}} (f)$.
We conjecture that for every endomorphism $f$ of the projective space the two Fatou sets coincide.

 \medskip
 
 There are two main results on the geometry of the Fatou set of an endomorphism of the complex projective space  of degree at least 2, see \cite{Sibony} for a complete survey.
 Every Fatou component is a Stein space \cite{FS2} and is hyperbolically imbedded in  $\mathbb{P}^N_\C$ in the sense of Kobayashi \cite{Ueda}.
  
\smallskip
 
Here we shall focus our attention on the  hyperbolicity properties of the harmonic Fatou components in the non-Archimedean setting. 
To motivate our next result, recall that  a subspace $\Omega$  of a complex analytic space $Y$ is hyperbolically imbedded 
if the Kobayashi distance on $\Omega$ does not degenerate towards its boundary  \cite{Kobbook,Langcomplex}.
If $\Omega$ is relatively compact in $Y$, then $\Omega$ is hyperbolically imbedded in $Y$ if and only if 
the family $\mathrm{Hol}(\D, \Omega)$ of holomorphic maps from the open unit disc $\D$ to $\Omega$
 is relatively locally compact in $\mathrm{Hol}(\D, Y)$, see \cite[\S II, Theorem 1.2]{Langcomplex}.

\smallskip

 In our context, we prove:

\begin{thmx}\label{thm c}
Let $f:\mathbb{P}^{N, \an} \to \mathbb{P}^{N, \an}$ be an endomorphism  of degree at least 2.
Let $\Omega$ be a connected component of the harmonic Fatou set  $F_{\mathrm{harm}} (f)$ of $f$, and let $U$ be any connected open subset of $\mathbb{P}^{1, \an}$. 

Then, every sequence of analytic maps 
 $g_n : U \to \Omega$ admits a subsequence $g_{n_j}$ that is pointwise converging to a continuous map $U \to \mathbb{P}^{N,an}$.
\end{thmx}

Note that in the non-Archimedean setting checking the normality for \emph{every} open subset $U$ of $\mathbb{P}^{1,\an}$ is stronger than just for the open unit disk, as opposed to the complex case, see \cite[Theorem 5.1.5]{Kobbook}.
 For instance, every sequence of analytic maps $f_n: \D \to \A^{1, \an} \setminus \{ 0 \}$ admits a subsequence converging to a continuous map,
  whereas this is not true if one replaces the source by the punctured open unit disk.

It remains open whether in Theorem \ref{thm c} one can take $U$ to be any basic tube.

\medskip

We have the following Picard-type result:

\begin{thmx}\label{thm d}
Let $\Omega$ be a connected component of the harmonic Fatou set  $F_{\mathrm{harm}} (f)$ 
 of an endomorphism $f: \mathbb{P}^{N, \an} \to \mathbb{P}^{N, \an}$ of degree $d\ge 2$.
Then every analytic map  from $\A^{1,\an}\setminus \{ 0 \}$ to $\Omega$ is constant.
\end{thmx}

\medskip

This paper is structured as follows. In Section \ref{section intro berko} we review some basic facts about Berkovich spaces  and summarize several results on universal points from \cite{Poineau} that will be needed in the sequel.
In section \ref{section polynomial} we prove a version of Theorem \ref{thm a} for polynomial maps of uniformly bounded degree.
In Section \ref{section mor} we describe the structure of the topological space that parametrizes the continuous maps that appear as pointwise limits of analytic maps between polydisks. 
Section \ref{section preuve montel}  comprises the proofs of Theorem \ref{thm a} and Theorem \ref{thm b}.
The properties  continuous maps that are limits of analytic maps are  studied in \S \ref{section weakly analytic}.
 Finally, in Section \S \ref{section dynamics} we give applications to dynamics of the previous results and prove Theorem \ref{thm c} and Theorem \ref{thm d}.

\subsection*{Acknowledgements}

I would like to thank my advisor Charles Favre for his constant support during the preparation of the paper.
I would also like to show my gratitude to Jérôme Poineau and Antoine Ducros for answering numerous questions, and to Junyi Xie for his attentive reading  and precious remarks.

This research was supported by the ERC grant Nonarcomp no. 307856.

\section{General facts on analytic spaces}\label{section intro berko}

Throughout this paper, $k$ is a field endowed with a  non-Archimedean complete norm $|.|$. 
We will \emph{always} assume that $k$ is nontrivially valued. 
 Except in \S \ref{sec:other base field}, $k$ will be algebraically closed.

We write $|k^\times|= \{ |x| :   x\in k^\times \} \subseteq \R_+$ for its value group and  $k^\circ = \{x\in k:  |x| \leq 1\}$ for its ring of integers. The latter is a local ring with maximal ideal $k^{\circ  \circ}= \{ x\in k: |x| < 1\}$.
The \emph{residue field of $k$} is  $\tilde{k}=k^\circ / k^{\circ  \circ}$.

\medskip

The basic reference for this section is Berkovich's original text \cite{Berk}. See also \cite{Temkinintro} for a more recent survey.

\subsection{Analytic spaces}

Pick a positive integer $N$ and an $N$-tuple of positive real numbers $r=(r_1, \cdots, r_N)$. Denote by $k\{r^{-1} T \}$ the set of power series $f=\sum_I a_I T^I$, $I=(i_1, \cdots, i_N)$, with coefficients $a_I\in k$ such that $|a_I| r^I \to 0$ as $|I|:= i_1+ \cdots + i_N$ tends to infinity.  The norm $\| \sum_I a_I T^I\| = \max_I |a_I|r^I$ makes $k\{r^{-1} T \}$ into a Banach $k$-algebra.
When $r=(1, \cdots, 1)$, the previous algebra is called the Tate algebra and we denote it by $\mathcal{T}_n$.

 Let $\varphi: \mathcal{B} \to \mathcal{A}$ be a morphism of Banach $k$-algebras.
 The residue norm on $\mathcal{B}/\mathrm{Ker} \varphi$  is defined by $|a| = \inf_{\varphi(b)=a} |b|$, and we say that $\varphi$ 
  is admissible if the residue norm is equivalent to the restriction  to the image of $\varphi$ of the norm on $\mathcal{A}$.
\medskip

A Banach $k$-algebra $\mathcal{A}$ is called \emph{affinoid} if there exists an admissible surjective morphism of $k$-algebras $k\{r^{-1} T \} \to \mathcal{A}$.
If  $r_i \in  |k^\times|$ for all $i$, then $\mathcal{A}$ is said to be strictly affinoid. 
\medskip

For any $k$-affinoid algebra  $\mathcal{A}$, we denote by $X=\mathcal{M(A)}$ the set of all multiplicative seminorms on $\mathcal{A}$ that are bounded by the norm $\| . \| $ on $\mathcal{A}$.
 Given $f\in \mathcal{A}$, its image under a seminorm $x\in \mathcal{M(A)}$ is denoted by $|f(x)| \in \R_+$. 
The set $\mathcal{M(A)}$ is called the \emph{analytic spectrum} of $\mathcal{A}$ and  is endowed with the weakest topology such that all the functions of the form $x\mapsto |f(x)|$ with $f\in \mathcal{A}$ are continuous. The resulting topological space $X$ is nonempty, compact Hausodorff \cite[Theorem 1.2.1]{Berk} and naturally carries a sheaf of analytic functions $\mathcal{O}_X$ such that $\mathcal{O}_X(X)= \mathcal{A}$, see \cite[\S 2.3]{Berk}). The locally ringed space $(X, \mathcal{O}_X)$  is called a $k$-affinoid space.
\medskip

Given a point $x \in X= \mathcal{M(A)}$, the fraction field of $\mathcal{A}/\mathrm{Ker}(x)$ naturally inherits from $x$ a norm extending the one on $k$.
Its completion is  the \emph{complete residue field at $x$} and denoted by $\mathcal{H}(x)$.
When $\mathcal{H}(x)$ is a finite extension of $k$ (or equivalently when $\mathcal{H}(x) = k$, since $k$ is supposed to be algebraically closed), we say that $x$ is \emph{rigid}.  The set $X(k)$ of rigid points of $X$  is dense in $X$. 
\medskip

A character on $\mathcal{A}$ is a bounded homomorphism $\mathcal{A} \to K$, where $K$ is any complete extension of $k$.
 Two characters $ \chi_1: \mathcal{A} \to K_1$ and $\chi_2 : \mathcal{A} \to K_2$ are equivalent
  if there exists a character $\chi: \mathcal{A} \to L$ and inclusions $i_1 : L \to K_1$ and $i_2 : L \to K_2$ such that $ i_1 \circ \chi = \chi_1$ and  $i_2 \circ \chi = \chi_2$.
  
 Composing the character $\mathcal{A} \to K$ with the norm on $K$ gives rise to a seminorm on $\mathcal{A}$ that is bounded, and thus corresponds to a point $ x \in \mathcal{M(A)}$. 
 Equivalent characters give rise to the same point. Conversely, every  point $x \in \mathcal{M(A)}$ induces a character $\chi_x: \mathcal{A} \to \mathcal{H}(x)$ in a natural way.
Any other  character $\mathcal{A} \to K$ giving rise to $x$ can be decomposed as $\mathcal{A} \to \mathcal{H}(x) \hookrightarrow K$. 
 
 \medskip

The closed polydisk  of dimension $N$ and polyradius $r = (r_1, \cdots, r_N)\in (\R^+_*)^N $ is defined to be $\bar{\D}^N (r) := \mathcal{M}(k\{r^{-1} T \})$.  
When $r = (1, \cdots, 1)$ we just write $\bar{\D}^N$, and when $N=1$ we denote it by $\bar{\D}$. 
The Gauss point $x_g\in\bar{\D}^N$ is the point associated to the norm 
$$|( \sum a_I T^I) (x_g)| :=  \max |a_I| . $$
\medskip

General analytic spaces are ringed spaces $(X, \mathcal{O}_X)$ obtained by gluing together affinoid spaces. Difficulties arise in the gluing construction as affinoid spaces are compact, and we refer to \cite{Berk,Berk2} for a precise definition. Analytic spaces are locally compact and locally path-connected \cite[Theorem 3.2.1]{Berk}.
 Given an analytic space $X$, we denote by $|X|$ its underlying topological space. 

\medskip

The following topological  result, due to  Poineau, will be systematically used throughout the paper:

\begin{theorem}[\cite{Poineau}]\label{seq compactness}
Every $k$-analytic  space $X$ is a Fréchet-Urysohn space. In particular, every compact subset of $X$ is sequentially compact.
\end{theorem}

In the following, we will always deal with \emph{good} analytic spaces,
which is formed by the subcategory of analytic spaces 
that are locally ringed spaces modelled on affinoid spaces. 
In other words, any point in a good analytic space admits a neighbourhood isomorphic to an affinoid space.

\medskip

For any point $x$ in a $k$-analytic space $X$, the stalk $\mathcal{O}_{X,x}$ is a local $k$-algebra with maximal ideal $\mathfrak{m}_x$. It inherits an absolute value extending the one on $k$, and the completion of  $\mathcal{O}_{X,x}/ \mathfrak{m}_x$  is again called the \emph{completed residue field of $x$} and denoted by $\mathcal{H}(x)$. In particular, when $X$ is an affinoid space, this definition coincides with the previous one.

\medskip

The open polydisk of dimension $N$ and polyradius  $r = (r_1, \ldots, r_N)\in (\R^+_*)^N $ is the set
$$\D^N_k(r) =\{ x \in \bar{\D}^N_k (r) : |T_i(x)| <r_i,  i=1, \ldots, N \}.$$ 
It can be naturally endowed with a structure 
of good analytic space by writing it as the  increasing union of $N$-dimensional polydisks $\D^N_k(\rho)$ whose radii $\rho = (\rho_1, \cdots, \rho_N)\in (\R^+_*)^N $ satisfy $\rho_i < r_i$ for all $i=1, \ldots, N$.

\subsection{Analytic maps}\label{section maps}

 A morphism of $k$-affinoid spaces $\mathcal{M(A)} \to \mathcal{M(B)}$ is by definition one induced by a bounded morphism of Banach $k$-algebras $\varphi: \mathcal{B} \to \mathcal{A}$. 
The fibre of  $\varphi^\sharp: \mathcal{M(A)} \to \mathcal{M(B)}$ over a point 
 $y \in \mathcal{M(B)}$  is isomorphic to $\mathcal{M(A} \hat{\otimes}_{\mathcal{B}} \mathcal{H}(y))$, see \S \ref{subsection universal} for the notion of complete tensor product.
  Indeed, let $y\in \mathcal{M(B)}$ and let $\chi_y: \mathcal{B}\to \mathcal{H}(y)$ be the associated character.
  By definition, a point $x \in \mathcal{M(A)}$ is mapped to $y$ if and only if the composition  $\mathcal{B} \stackrel{\varphi}{\to} \mathcal{A}\to \mathcal{H}(x)$ factors through $\mathcal{H}(y)$, which is equivalent to the character $\chi_x$ factorizing through the $\mathcal{B}$-algebra morphism $\mathcal{A} \hat{\otimes}_{\mathcal{B}} \mathcal{H}(y) \to \mathcal{H}(x)$.
Pick  $ x \in \mathcal{M(A} \hat{\otimes}_{\mathcal{B}} \mathcal{H}(y))$ and let $\mathcal{A} \hat{\otimes}_{\mathcal{B}} \mathcal{H}(y) \to \mathcal{H}(x)$ be the associated character. 
The latter is equivalent to the data of morphisms $\mathcal{H}(y) \to  \mathcal{H}(x)$ and $\mathcal{A} \to \mathcal{H}(x)$ such that the composition $\mathcal{B} \stackrel{\varphi}{\to} \mathcal{A}\to \mathcal{H}(x)$ equals $\mathcal{B} \to \mathcal{H}(y) \to \mathcal{H}(x)$. In other words, the image of $x$ in $\mathcal{M(A)}$ is mapped to $y$ by $\varphi$.

\medskip

A morphism $\mathcal{M(A)} \to \mathcal{M(B)}$  is  a  \emph{closed immersion} when $\varphi$ is surjective and admissible.

\medskip

A surjective morphism $\varphi: \mathcal{T}_N \to \mathcal{A}$ is called \emph{distinguished} if the quotient norm $|.|_\varphi$ induced by $\varphi$ agrees with the supremum norm on $\mathcal{A}$, see \cite[\S 6.4.3]{BGR}.
 We say that $\mathcal{A}$ is distinguished if such an epimorphism exists.

  It can be shown that over an algebraically closed field $k$, every reduced algebra (i.e. without non-trivial nilpotents) is distinguished \cite[Theorem 6.4.3/1]{BGR}. The key property of distinguished epimorphisms is that the reduction $\widetilde{\mathcal{A}}$ is isomorphic to the quotient $\widetilde{\mathcal{T}}_N / \widetilde{\mathrm{ker} (\varphi)}$.

\medskip

From the definition one obtains the following useful result:

\begin{prop}\label{prop distinguished}
Let $X$ be a $k$-affinoid space and let $X \to \bar{\D}^N$ be a closed immersion induced by a distinguished morphism of Banach algebras.
Then, every analytic map on $X$ with values in a polydisk $\bar{\D}^M$ extends to an analytic map $\bar{\D}^N \to \bar{\D}^M$.
\end{prop}

\begin{proof}
Let $\mathcal{A}$ be the underlying affinoid algebra of $X$.
Pick an analytic map $f: X \to \bar{\D}^M$, 
which by definition is given by elements $f_1 , \ldots , f_M \in \mathcal{A}$ with $|f_i |_{\sup} \le 1$.
Fix a   distinguished epimorphism $\T_N \to \mathcal{A}$.
For $l= 1, \ldots , M$, we may lift $f_l$ to an element $g_l$ in $\T_N$ having the same norm.
The resulting analytic map $g = (g_1, \ldots , g_M): \bar{\D}^N \to \bar{\D}^M$ agrees with $f$ on the affinoid space $X$.
\end{proof}

Given any two $k$-analytic spaces $X$ and $Y$, we let $\mathrm{Mor}_k(X, Y)$ be the set of all  analytic maps from $X$ to $Y$.

\subsection{Analytification of algebraic varieties}

To every algebraic variety $X$ over $k$ one can associate a $k$-analytic space $X^\an$ in a functorial way; see \cite[\S 3.4]{Berk} for a detailed construction. 

In the case of an affine variety $X= \mathrm{Spec}(A)$, where $A$ is a finitely generated  $k$-algebra, then the set $X^{\an}$ consists of all the multiplicative seminorms on $A$ whose restriction to $k$ coincides with the norm on $k$. This set is endowed with the weakest topology such that all the maps of the form $x \in X^{\an} \mapsto |f(x)|$ with $f\in A$ are continuous. 

Observe that any $k$-point $x\in X$ corresponds to a morphism of $k$-algebras $A\to k$ and its composition with the norm on $k$
defines a rigid point in $X^{\an}$. Since $k$ is algebraically closed, one obtains in this way an identification of the set of closed points in $X$ with the set of rigid points in $X^{\an}$.

Let $X$ be a general algebraic variety  and fix  an affine open cover.
The analytification of a general algebraic variety $X$  is obtained by glueing together the analytification of its affine charts in natural way. Analytifications of algebraic varieties are good analytic spaces, and closed points are in natural bijection with rigid points as in the affine case.

\subsection{Boundary and interior}\label{section boundary}

Any $k$-analytic space $X$ comes with natural notions of boundary and interior. We shall restrict our attention to good $k$-analytic spaces.
\medskip

A point $x$ in an affinoid space $X$ lies in the \emph{interior} of $X$ if there exists a closed immersion $\varphi: X \to \bar{\D}^N (r)$ for some polyradius $r$ and some integer $N$ such that $\varphi (x)$ lies in the open polydisk $\D^N (r)$.

If $X$ is a good analytic space, a point $x$  belongs to its interior if it admits an affinoid neighbourhood $U$ such that $x$ belongs to the interior of $U$.
We let $\Int(X)$ be  the open set consisting of all the interior points in $X$. 
Its complement $\partial (X)$ is called the boundary of $X$. It is a closed subset of $X$.

The analytification of any algebraic variety is boundaryless.

\medskip

In the remaining of this section, we explain how to compute the interior of a strictly $k$-affinoid space $X=\mathcal{M(A)}$. Recall that the spectral radius of $f \in \mathcal{A}$ is defined by
$$\rho (f) = \lim_{n \to \infty} \| f^n\|^{1/n},$$
where $\| \cdot \|$ is the Banach norm on $\mathcal{A}$. 
The supremum seminorm on  $\mathcal{A}$ is 
defined by $|f|_{\sup} := \sup \{ |f(x)| : x \in \mathcal{M(A)} \} $ for $f \in \mathcal{A}$.
The spectral radius and the supremum seminorm agree \cite[Theorem 1.3.1]{Berk}.
 
When $\mathcal{A}$ is reduced, then $\rho$ is  a norm equivalent to $\|\cdot \|$. 
The set $\mathcal{A}^{\circ} = \{f\in \mathcal{A}: \rho (f) \le 1 \}$ is a subring of $\mathcal{A}$ and   
$\mathcal{A}^{\circ \circ} = \{f\in \mathcal{A}: \rho (f) < 1 \}$ an ideal.
The reduction of $\mathcal{A}$ is  then defined as $\widetilde{\mathcal{A}}:= \mathcal{A}^{\circ}/\mathcal{A}^{\circ \circ}$, and 
the reduction of $X$ is $\widetilde{X} = \mathrm{Spec}(\widetilde{A})$.

Observe that  Noether's normalization Lemma \cite[Corollary  6.1.2/2]{BGR} implies that 
for any  strictly $k$-affinoid algebra $\mathcal{A}$, the reduction 
 $\widetilde{\mathcal{A}}$ is a finitely generated $\tilde{k}$-algebra, and thus  $\widetilde{X}$ is an affine variety over the residue field $\tilde{k}$.
The reduction of the closed polydisk $\bar{\D}^N_k$ is the affine space $\A^N_{\tilde{k}}$.

\medskip

 The reduction map $\mathrm{red}: X \to \widetilde{X}$ is defined as follows. Every bounded morphism of Banach $k$-algebras $\mathcal{A} \to \mathcal{B}$ induces a morphism between their reductions $\widetilde{\mathcal{A}} \to \widetilde{\mathcal{B}}$. 
In particular, from the character $\chi_x: \mathcal{A} \to \mathcal{H}(x)$
associated to a point $x\in X$ we obtain a  $\tilde{k}$-algebra morphism $\widetilde{\chi}_x: \widetilde{\mathcal{A}} \to \widetilde{\mathcal{H}(x)}$. We set $\mathrm{red}(x) := \mathrm{Ker}(\widetilde{\chi}_x)$. This map is anticontinuous for the Zariski topology, meaning that the inverse image of a closed set is an open set.

\begin{lema}\label{interieur - points fermes}
Let $X$ be a strictly $k$-affinoid space. Then,
\begin{equation*}
\Int(X) = \{ x\in X: \mathrm{red}(x) \mbox{ is a closed point}\}.
\end{equation*}
\end{lema}

\begin{proof}
Let $\varphi: X \to \bar{\D}^N$ be a closed immersion. The following diagram is commutative by construction:
\begin{equation*}
 \xymatrix{
X \ar[d]_{\mathrm{red}} \ar[r]^{\varphi} & \bar{\D}^N \ar[d]^{\mathrm{red}}\\
\widetilde{X}  \ar[r]^{\widetilde{\varphi}} &\A^N_{\tilde{k}}\\
}
\end{equation*}
Let $\mathcal{A}$ be the underlying affinoid algebra of $X$ and pick any $x\in X$. If  its reduction  $\tilde{x} = \mathrm{red} (x)$  is a closed point then so is $\widetilde{\varphi}( \tilde{x})$.
The inverse image of $\widetilde{\varphi}(\tilde{x})$ is isomorphic to an open polydisk.  Up to  composing $\varphi$ with an automorphism of $\bar{\D}^N$, we may assume that $\mathrm{red}^{-1} (\widetilde{\varphi}(\tilde{x}))$ is isomorphic to
$\D^N$.
The commutativity of the diagram implies that $\varphi(x)$ lies in $\D^N$.

 Pick a point $x \in \Int (X)$.
By \cite[Proposition 2.5.2]{Berk}, the image of the
morphism of $\tilde{k}$-algebras $\widetilde{\chi}_x: \widetilde{\mathcal{A}} \to \widetilde{\mathcal{H}}(x)$ induced by  $\chi_x$ is integral over $\tilde{k}$. This implies that $\tilde{\chi}_x(\widetilde{\mathcal{A}})\simeq\widetilde{\mathcal{A}}/\mathrm{Ker}(\widetilde{\chi}_x)$ is a field.
 Thus, $\tilde{x}$ is a closed point of $\widetilde{X}$.
\end{proof}

\begin{prop}\label{interior- finite morphism}
Let  $X= \mathcal{M(A)}$ and $Y=\mathcal{M(B)}$ be  $k$-affinoid spaces, and let $f:X \to Y$ be a finite morphism. Then, $\Int(X)= f^{-1}(\Int (Y))$.
\end{prop}

This result is a consequence from \cite[Proposition 2.5.8]{Berk} and \cite[Corollary 2.5.13]{Berk}. Here we give a proof in the strictly $k$-affinoid case. 

\begin{proof}
We prove the result only in the strictly affinoid case. In order to adapt this proof to the general one, one needs to use Temkin's graded reduction of affinoid algebras (\cite{Temkinproperties1,Temkinproperties2}). 

The morphism $f: X \to Y$ induces the following commutative diagram:
\begin{equation*}
 \xymatrix{
X \ar[d]_{\mathrm{red}} \ar[r]^{f} & Y \ar[d]^{\mathrm{red}}\\
\mathrm{Spec}(\widetilde{\mathcal{A}})  \ar[r]^{\tilde{f}} &\mathrm{Spec}(\widetilde{\mathcal{B}})\\
}
\end{equation*} 
Let $x$ be a point in $\Int(X)$. By Lemma \ref{interieur - points fermes}, its image $f(x)$ belongs to $\Int (Y)$.

Let now $x \in X$ be such that $f(x) = y$ lies in $\Int(Y)$. By the previous lemma, we have to show that $\mathrm{red}(x)$ is a closed point of $\widetilde{X}$.
 Consider the ring homomorphism $\varphi: \widetilde{\mathcal{B}} \to \widetilde{\mathcal{A}}$ inducing $\tilde{f}$. It induces a morphism $\varphi^\prime: \widetilde{\mathcal{B}}/ \mathrm{ker}(\widetilde{\chi}_y) \to \widetilde{\mathcal{A}}/ \mathrm{ker}(\widetilde{\chi}_x)$, as the diagram above is commutative.
Observe that $\varphi$ is integral, since it is finite (\cite[Theorem 6.3.5/1]{BGR}), and thus $\varphi^\prime$ is also integral.
As $y \in \Int(Y)$, by Lemma \ref{interieur - points fermes} the quotient $\widetilde{\mathcal{B}}/ \mathrm{ker}(\widetilde{\chi}_y)$ is a field.
This implies that $\widetilde{\mathcal{A}}/\ker(\chi_x)$ is a field and thus that $\mathrm{red}(x)$ is a closed point.
\end{proof}

\subsection{Basic tubes}\label{section tubes}

We introduce the following terminology.

\begin{defini}
A $k$-analytic space  $X$ is called a basic tube  if there exists a reduced equidimensional strictly $k$-affinoid space $\hat{X}$ and a  closed point $\tilde{x}$ in its reduction  such that $X$ is isomorphic to $\mathrm{red}^{-1}(\tilde{x})$.
\end{defini}

 By convention, a basic tube is reduced.

\begin{theorem}\label{thm:basic tube}
A basic tube is connected.
\end{theorem}

The fact that any basic tube over an algebraically closed field is connected is a deep theorem due to
\cite{Bosch}, which was generalized to arbitrary base fields in  \cite{Poineaucomposantes}.

\begin{example}
Let $a_1, \cdots, a_m$ be type II points in $\mathbb{P}^{1, \an}$ as defined in Berkovich classification of points in  $\mathbb{P}^{1, \an}$, see \cite[\S 1.4.4]{Berk}.
Then every connected component of $\mathbb{P}^{1, \an} \setminus \{ a_1, \cdots, a_m\}$ is a basic tube. 
\end{example}

\begin{prop}\label{prop tube cc aff}
A $k$-analytic space $X$ is a basic tube if and only if it is isomorphic to  a connected component of the interior of some equidimensional strictly $k$-affinoid space.
\end{prop}

\begin{obs}
Every good reduced boundaryless $k$-analytic space has a basis of open neighbourhooods that are basic tubes.
\end{obs}

\begin{proof}
Let $V$ be any connected component of the interior of an equidimensional strictly $k$-affinoid space  $\hat{X}$.
By Lemma \ref{interieur - points fermes}, $\red (V)$ is contained in the set of closed points of the reduction of $\hat{X}$.
If $\red (V)$ contains at least two distinct points, then $V$ can be written as a disjoint union of nonempty open sets, contradicting the connectednes. Hence, $\red (V)$ is a singleton.

Let conversely  $X = \red^{-1} (\tilde{x})$ be a basic tube, where $\tilde{x}$ is a closed point in the reduction of an equidimensional strictly $k$-affinoid space $\hat{X}$. Clearly, $X$ is contained in some connected component $V$ of $\Int(\hat{X})$.
The previous argument shows that $\red (V) = \{ \tilde{x}\}$.
\end{proof}

Recall that a topological space is $\sigma$-compact if it is the union of countably many compact subspaces.  For instance,   open Berkovich polydisks or the analytification of an algebraic variety are $\sigma$-compact spaces. Observe that there exist simple examples of $k$-analytic spaces which are not
$\sigma$-compact, e.g. the closed unit disk of dimension $N
\ge 2$ with the Gauss point removed over a base field $k$ with uncountable reduction $\tilde{k}$.

\begin{prop}\label{basic tube intersection}
For every basic tube $X$ there exist a strictly $k$-affinoid space $\hat{X}$ and a distinguished closed immersion into some closed polydisk $\hat{X} \to \bar{\D}^N$ such that $X$ is isomorphic to $\hat{X} \cap \D^N$.
\end{prop}

In particular, $X$ is boundaryless and $\sigma$-compact.

\begin{proof}
Let $\hat{X}= \mathcal{M(A)}$ be an equidimensional reduced $k$-affinoid space and let $\tilde{x}$ be a closed point  in its reduction such that $\red^{-1}(\tilde{x})$ is isomorphic to $X$. 
As  $k$ is algebraically closed and $\mathcal{A}$ is reduced, there exists a distinguished closed immersion $\varphi: \hat{X} \to \bar{\D}^N$, see \cite[Theorem 6.4.3/1]{BGR}.
Hence,   $\widetilde{\mathcal{A}}$ is isomorphic to $\tilde{k}[T_1, \cdots , T_N] / \widetilde{\ker (\varphi)}$ by  \cite[Corollary 6.4.3/5]{BGR}.

The induced morphism $\mathrm{Spec}(\widetilde{\mathcal{A}}) \to \A^N_{\tilde{k}}$ is a closed immersion by \cite[Proposition 6.4.3/3]{BGR}, since $\varphi$ is distinguished.
We may assume that $\tilde{x}$ is mapped  to $0$. We conclude that $x$ is mapped to a point in  $\red^{-1}(0)$, which is isomorphic to $\D^N$. % by \cite[Proposition 2.2]{BL1}.
\end{proof}

\subsection{Universal points and base changes}\label{subsection universal}

Let $\mathcal{A}$ and $\mathcal{B}$ be two Banach $k$-algebras and denote by $|.|_{\mathcal{A}}$ and $ |.|_{\mathcal{B}}$ their respective Banach norms. On the tensor product $\mathcal{A} \otimes_k \mathcal{B}$ we have the seminorm that associates to every $f \in \mathcal{A} \otimes_k \mathcal{B}$ the quantity
$$|| f || = \inf \max |a_i|_{\mathcal{A}} \cdot |b_i|_{\mathcal{B}},$$
where the infimum is taken over all the possible expressions of $f$ of the form $f= \sum_i a_i \otimes b_i$ with $a_i \in \mathcal{A}$ and $b_i \in \mathcal{B}$.
 The seminorm $||.||$ induces the \emph{tensor norm} on the quotient  $\mathcal{A}\otimes_k \mathcal{B} / \{ ||f||=0\}$, whose completion is a Banach $k$-algebra satisfying a suitable natural universal property. This algebra is called the \emph{complete tensor product} of $\mathcal{A}$ and $\mathcal{B}$ and  we denote it by $\mathcal{A}\hat{\otimes}_k \mathcal{B}$, see \cite[\S 2.1.7]{BGR}). 

\medskip

Given a $k$-affinoid algebra $\mathcal{A}$ and a complete extension $K$ of $k$, the $K$-algebra $\mathcal{A} \hat{\otimes}_k K$ is in fact $K$-affinoid. 
One defines the scalar extension of the $k$-affinoid space $X= \mathcal{M(A)}$ by $K$ as the  $K$-affinoid space $X_K := \mathcal{M}(\mathcal{A} \hat{\otimes}_k K)$. The natural morphism $\mathcal{A}\to \mathcal{A}\hat{\otimes}_k K$ induces a base change morphism  
 $\pi_{K/k}: X_K \to X$ which is continuous and surjective. This construction can be done similarly for general $k$-analytic spaces.
\medskip

Recall the following definition from \cite{Berk,Poineau}:
\begin{defini}
Let $X$ be a $k$-analytic space.
 A point $x$ in $X$  is \emph{universal} if for every complete extension $K$ of $k$ the tensor norm on  $\mathcal{H}(x) \hat{\otimes}_k K$ is multiplicative.
\end{defini}

The key feature of universal points is that they can be canonically lifted to any scalar extension. To explain this fact we may 
suppose that $X$  is an affinoid space with underlying algebra $\mathcal{A}$. Pick any universal point $x\in X$ and fix any complete extension $K$ of $k$.
The $k$-algebra morphism $\mathcal{A} \to \mathcal{H}(x)$ corresponding to  the point $x$  induces a $K$-algebra morphism $\mathcal{A}\hat{\otimes}_k K \to \mathcal{H}(x)\hat{\otimes}_k K$. 

Since $x$ is universal, the tensor norm on $\mathcal{H}(x) \hat{\otimes}_k K$ is multiplicative, and so  the composition of $\mathcal{A}\hat{\otimes}_k K \to \mathcal{H}(x)\hat{\otimes}_k K$  with the tensor norm defines a point in $X_K$.
The point in $X_K$ obtained by these means is denoted by $\sigma_{K/k}(x)$. 

Observe that if $x\in X$ is rigid, then so is  $\sigma_{K/k}(x)$, and that $\sigma_{K/k}$ is a section of $\pi_{K/k}$ on the set of universal points of $X$. 

\begin{theorem}[\cite{Poineau}]
Let $k$ be an algebraically closed complete field and $X$ a $k$-analytic space. Then,
every point $x \in X$ is universal, and the map $\sigma_{K/k}: X \to X_K$ is continuous. 
\end{theorem}

We conclude this section by recalling the following construction. 
 \begin{lema}\label{fibres}
 Let $X$ be a $k$-analytic space and $x$ a point in $X$.  
 Then for every complete extension $K$ of $\mathcal{H}(x)$, the fibre $\pi_{K/k}^{-1}(x)$ contains a rigid point.
 \end{lema}

 \begin{proof}
 Pick a point $x \in X$.
We may suppose $K = \mathcal{H}(x)$. 
Since the statement is local at $x$, we may replace $X$ by any  affinoid domain of $X$ containing  $x$. 
Denote by $\mathcal{A}$ the underlying $k$-affinoid algebra.
Consider the character $\chi_x:\mathcal{A} \to \mathcal{H}(x)$. The morphism $\mathcal{A}\hat{\otimes}_k \mathcal{H}(x) \to \mathcal{H}(x) $ sending  $f\otimes a$ to $\chi_x(f)\cdot a$
 is by definition  a rigid point in $X_{\mathcal{H}(x)}$ lying over $x$. 
 \end{proof}

We shall denote by $\tau(x) \in X_{\mathcal{H}(x)}$ the rigid point lying over $x \in X$ obtained in the previous proof.
 This point $\tau(x)$ is not to be confused with $\sigma_{K/k}(x)$.

\section{Polynomial maps of bounded degree}\label{section polynomial}

As a first step in proving Theorem \ref{thm a}, we deal with the case of sequences of polynomial maps of bounded degree. 

Throughout this section, we fix  integers $r,s, \delta >0$ and assume that the base field $k$ is algebraically closed.

The result we aim to show is the following:

\begin{prop}\label{montel poly} 
Let $k$ be an algebraically closed non-Archimedean complete field.
Let $f_n: \A^{r,\an} \to \A^{s,\an}$ be a sequence of polynomial maps of 
uniformly bounded degree satisfying $f_n (\D^r) \subset \bar{\D}^s $.
Then, there exists  a subsequence that is converging pointwise to a continuous map $f: \A^{r,\an} \to \A^{s,\an}$.
\end{prop}

\subsection{Parametrization of polynomial maps of uniformly bounded degree}

In order to prove this  theorem, we reinterpret polynomial maps between analytic affine spaces as rigid points in a closed polydisk.

\medskip

Given a multi-index $I=(i_1, \cdots, i_r)$, denote by $| I| = \max_j i_j$.

\bigskip

 Every polynomial map $f : \A^{r,\an} \to \A^{s,\an}$ of degree at most $\delta^r$ where $\delta \in \N^*$ satisfying $f(\D^r ) \subseteq \bar{\D}^s$ is of the form
 \begin{equation*}
  f =(f_1, \cdots, f_s) = \left( \sum_{|I| \le \delta} a_{1,I} T^I, \cdots, \sum_{|I| \le \delta} a_{s,I} T^I \right) , 
  \end{equation*}
   with  $|a_{l,I}|\leq 1$.
    Thus, the point 
\begin{equation}\label{eq:alpha delta}
\al = \alpha (f) := \left( (a_{1,I})_{\vert I \vert \leq \delta} , \cdots, (a_{s,I})_{\vert I \vert \leq \delta}\right)
\end{equation}
can be realized as rigid point in the (Berkovich) analytic space $\bar{\D}^{s(\delta+1)^r}$. 

\medskip

Additionally, to every   not-necessarily rigid point  $\alpha$ in $\bar{\D}^{s(\delta+1)^r}$ we shall associate a continuous map
$$P_\alpha = P^{r,s}_\alpha : \A^{r,\an} \to \A^{s, \an}$$ as follows.
Consider first the analytic map $\Phi:  \bar{\D}^{s(\delta + 1)^r} \times \A^{r,\an} \to \A^{s,\an}$, given by the  $k$-algebra morphism
\begin{eqnarray*}
k[T_1, \ldots, T_s] & \to & k[T_1, \ldots, T_r]\{ (a_{1, I})_{|I| \le \delta}, \cdots, (a_{s,I})_{|I| \le \delta} \}\\
T_l & \mapsto & \sum _{|I| \le \delta} a_{l,I} T^I.
\end{eqnarray*} 
Next, consider the projection $\pi_1: \bar{\D}^{s(\delta + 1)^r} \times \A^{r,\an} \to \bar{\D}^{s(\delta + 1)^r}$.
 The fibre over the point $\alpha \in \bar{\D}^{s(\delta + 1)^r}$ is isomorphic to $\A^{r,\an}_{\mathcal{H}(\alpha)}$  (cf. \S \ref{section maps})). 
   Recall  from \S \ref{section intro berko}  that the point $\alpha \in \bar{\D}^{s(\delta+1)^r}$ is associated to the character
$\chi_\alpha: k\{ (a_{1, I})_{|I| \le \delta}, \ldots, (a_{s,I})_{|I| \le \delta} \} \to \mathcal{H}(\alpha)$. Set $K:= \mathcal{H}(\alpha)$.
 The inclusion $\iota_K: \A^{r,\an}_K \to \bar{\D}^{s(\delta + 1)^r}_k \times \A^{r,\an}_k$ is given by
\begin{eqnarray*}
k[T_1, \cdots, T_r]\{ (a_{1, I})_{|I| \le \delta}, \ldots, (a_{s,I})_{|I| \le \delta}\} & \to & K[T_1, \ldots, T_r]\\
T_i & \mapsto & T_i\\
a_{l,I} & \mapsto & \chi_\alpha ( a_{l,I})~.
\end{eqnarray*}
Finally, for every $z \in \A^{r, \an}$ we set:
\begin{equation}\label{eq def Pa}
P_\alpha (z) = \Phi \circ \iota_K \circ \sigma_{K/k} (z)~,
\end{equation}
where $\sigma_{K/k} : \A^{r, \an}_k \to \A^{r, \an}_K$ is the canonical lift discussed in \S \ref{subsection universal}.
The map $P_\alpha : \A^{r,\an} \to \A^{s, \an}$ is clearly continuous. 
Explicitely, given a polynomial $g = \sum_J g_J T^J \in k[T_1, \cdots, T_s]$ and a point $z \in \A^{r, \an}$, we have
\begin{equation} \label{eq: P alpha}
\left|g (P_\alpha (z)) \right|= \left| \left( \sum_{J \in \N^s} g_J \prod_{l=1}^s \left( \sum_{|I| \le \delta} \chi_\alpha (a_{l,I}) T^I \right)^{j_l} \right) \sigma_{K/k} (z)\right| ~.
\end{equation}
To emphasize the fact that  $\bar{\D}^{s(\delta+1)^r}$ parametrizes polynomial  maps of degree $\delta$, we shall denote it from now on by $\mor^{r,s}_\delta$. 
For $r,s$ and $\delta \in \N$ fixed, we have thus constructed a map
\begin{eqnarray*}
\ev: \mathrm{Mor}_\delta^{r,s} & \rightarrow & \mathcal{C}^0 ( \A^{r, \an}, \A^{s, \an})\\
\alpha & \mapsto & \ev(\alpha) := P_\alpha ~.
\end{eqnarray*}

\subsection{Remarks on the map $\ev$}
~\smallskip

\noindent{\bf1}.
The assignment $$(\alpha, z) \mapsto \ev(\al)(z)$$ does not define a continuous map on $|\mor_\delta^{r,s} | \times |\A^{r, \an}|$. 
This phenomenon already appears when $r=s=\delta=1$. 

Indeed, suppose by contradiction that there exists a continuous map 
$ \varphi: |\bar{\D}^2| \times |\A^{1, \an}| \to | \A^{1, \an} |$ such that $\varphi \left( (\al_0, \al_1) ,z \right) = \al_0 + \al_1 z$ for any $\al_0, \al_1,  z \in k$ and $|z|\le1$.
 Pick any sequence of points $\zeta_n \in k$ such that $|\zeta_n | = 1$ and $|\zeta_n - \zeta_m|=1 $ for $n \neq m$. 
 Both the sequences $\{ \zeta_n\} $ and $\{ - \zeta_n\} $ converge  to the Gauss point $x_g$. 
 We compute:
\begin{equation*}
\lim_n \varphi \left( (\zeta_n , 1),  \zeta_n \right) = \lim_n \varphi \left( (\zeta_n , 1), - \zeta_n \right)  = \varphi \left( (x_g , 1), x_g \right)  =  x_g~.
\end{equation*} 
However, 
 we have  that $\varphi \left( (\zeta_n , 1), - \zeta_n \right)   = 0$ for all $n$,     contradicting the continuity of $\varphi$.

\noindent{\bf2}.
In general, the map
\begin{eqnarray*}
\ev: \mathrm{Mor}_\delta^{r,s} & \rightarrow & \mathcal{C}^0 ( \A^{r, \an}, \A^{s, \an})\\
\alpha & \mapsto & \ev(\alpha) 
\end{eqnarray*}
is not injective.  This  already occurs in the case $r = s = 1$ for affine maps.
\medskip

Indeed, let $r=s=\delta=1$. 
The space $\mor_1^{1,1}$ is naturally isomorphic to the polydisk $\bar{\D}^2$. 
Denote by $p_0$ and $p_1$ the first and second projections $\mathrm{Mor}^{1,1}_1 \to  \mathrm{Mor}^{1,1}_0$.
Pick two points $\alpha, \alpha^\prime \in \mor_1^{1,1}$ such that $p_0(\al) = p_0 (\al^\prime) = x_g \in \bar{\D}$. 
As seen in \S \ref{section maps},  the fibre $p_0^{-1} (x_g)$ is naturally homeomorphic to $\bar{\D}_{\mathcal{H}(x_g)}$, and so the points $\alpha$ and $\al^\prime$ correspond to points $\al_1, \al_1^\prime \in \bar{\D}_{\mathcal{H}(x_g)}$ respectively.
Write $K= \mathcal{H}(x_g)$ for simplicity, and recall that $K$ is a non-trivial extension of $k$ that contains the field of rational functions in one variable $k(S)$ as a dense subset.
Assume that both  $\al_1$ and $\al^\prime_1$ are the rigid points in $\bar{\D}_K$ given by $\al_1 = Q(S) = q_0 + q_1 S + q_2 S^2$ and $\al^\prime_1 = Q^\prime (S) = q_0 + q_1 S + q_2^\prime S^2$,
 with $q_2 \neq q_2^\prime$ and
$|q_2| = |q_2^\prime |$.

We claim that $\ev(\al) = \ev (\al^\prime)$. It suffices to check that  they agree on the set of rigid points.
Indeed, pick any $z \in \A^{1,\an}(k)$. 
Following Berkovich's classification of the points in the disk \cite[\S 1.4.4]{Berk},
the point $\ev(\al)(z)$ corresponds to the closed ball in $k$ centered at $z q_0$ and of radius $\max \{ | 1+ q_1 z|, |q_2 z| \}$.
Since  $|q_2| = |q_2^\prime|$, we conclude that $\ev(\al)(z) = \ev(\al^\prime)(z)$.

\subsection{Proof of Proposition  \ref{montel poly}}

Consider a sequence of polynomial  maps $f_n : \A^{r, \an} \to \A^{s,\an}$  of degree at most $\delta \in \N$ satisfying $f_n (\D^r) \subset \bar{\D}^s$.

For every $n \in \N$, let $\al_n$ be the  rigid point in the polydisk $\bar{\D}^{s(\delta + 1)^r}$ corresponding to the mapping $f_n$, as constructed above.
The polydisk $\bar{\D}^{s(\delta+1)^r}$ is sequentially compact by Theorem  \ref{seq compactness}, therefore we may find a subsequence $\{ \alpha_{n_j} \}_{n_j}$ converging to some point $\al \in \bar{\D}^{s(\delta+1)^r}$.
Recall that this limit point defines a continuous map $\ev(\al) : \A^{r,\an} \to \A^{s,\an}$.

It remains to verify that $\ev(\al)$ is the pointwise limit of the subsequence $\{ f_{n_j} \}$. 
 Observe that this is equivalent to checking that for every $z\in \A^{r,\an}$ and every  $g\in k[T_1, \ldots, T_s]$, the sequence of real numbers $\{ |g(f_{n_j}(z))| \}_{n\in\N}$ converges to $| g (\ev(\al)(z) ) |$.

If $z$ is a non-rigid point in $\A^{r,\an}$, 
we make a base change by $\mathcal{H}(z)$ and take a rigid point  $x\in \D^r_{\mathcal{H}(z)}$ lying over $z$ (see Lemma \ref{fibres}).
 The maps $f_{n_j}$ induce analytic maps $\A^{r,\an}_{\mathcal{H}(z)} \to \A^{s,\an}_{\mathcal{H}(z)}$ and $g$ defines an analytic function on $\A^{r,\an}_{\mathcal{H}(z)}$. 
By definition,
 $$|g(f_{n_j}(z))| = |g(f_{n_j} (\pi_{\mathcal{H}(z) / k} (x))| = |g(f_{n_j}(x))|~,$$ so that $ |g(f_{n_j}(z))|$ converges if and only if $|g(f_{n_j}(x))|$ converges.
 
 Similarly, $\ev(\al)$ defines a continuous map $\A^{r,\an}_{\mathcal{H}(z)} \to \A^{s,\an}_{\mathcal{H}(z)}$. 
  Indeed, recall from \eqref{eq def Pa} that $\ev(\al) = \Phi \circ \iota_{\mathcal{H}(\al)} \circ \sigma_{\mathcal{H}(\al)/k}$.
 As $\Phi$ is $k$-analytic, it induces a $\mathcal{H}(z)$-analytic map $\bar{\D}^{s(\delta + 1)^r }_{\mathcal{H}(z)}\times \A^{r,\an}_{\mathcal{H}(z)} \to \A^{s,\an}_{\mathcal{H}(z)}$ that we shall also denote by $\Phi$.

Denote by $L$ the complete residue field $\mathcal{H}(\sigma_{\mathcal{H}(z) / k } (\al))$, which is a complete extension of $\mathcal{H}(z)$.
Moreover, we claim that it is also a complete extension of $\mathcal{H}(\al)$.
In order to see this, notice that 
$\ker (\al) = k\{ a_{l,I}\}_{|I| \le \delta, 1 \le l \le s}  \cap \ker (\sigma_{\mathcal{H}(z) / k } (\al))$.
Thus, we have inclusions 
$$k\{ a_{l,I}\}_{|I| \le \delta, 1 \le l \le s} / \ker (\al) \subset \mathcal{H}(z)\{ a_{l,I}\}_{|I| \le \delta, 1 \le l \le s} / \ker (\sigma_{\mathcal{H}(z) / k } (\al))~,$$
and so $\mathcal{H}(\sigma_{\mathcal{H}(z) / k } (\al))$ is a complete extension of $\mathcal{H}(\al)$.

Consider next the inclusion $\iota_L : \A^{r,\an}_L \to \bar{\D}^{s(\delta + 1)^r }_{\mathcal{H}(z)}\times \A^{r,\an}_{\mathcal{H}(z)} $ given by the inclusion of the fibre of the first projection over the point $\sigma_{\mathcal{H}(z) / k } (\al)$.
We  obtain that $\ev(\al)$ induces the continuous map
$\Phi \circ \iota_L \circ \sigma_{L / \mathcal{H}(z)}$. By construction, we see that 
$$|g(\ev(\al)(z))| = |g(\ev(\al)(\pi_{\mathcal{H}(z) / k} (x))| = |g(\ev(\al)(x))|~.$$

\medskip

We may thus assume that $z$ is rigid. 
Let $g = \sum_{J \in \N^s} g_J T^J$ be a polynomial of degree $d$.
Denoting $f_{n_j} = (f_1^{(n_j)}, \ldots, f_s^{(n_j)} )$, we have:
\begin{eqnarray*}
\left| g(f_{n_j} (z))\right|  & = &  
 \left| \sum_{|J|\leq d} g_{J} \prod_{l=1}^s \left(f_l^{(n_j)}(z)\right)^{j_l} \right|  = \\
 & = &
  \left| \sum_{|J|\leq d} g_{J}\prod_{l=1}^s \left( \sum_{|I|\leq \delta} a_{l,I}^{(n_j)}z^I \right)^{j_l}\right| = (*)~.
\end{eqnarray*}
Taking the polynomial in  $s(\delta+1)^r$-variables 
\begin{equation}\label{polynomial R}
R:= \sum_{|J|\leq d} g_J \prod_{l=1}^s \left( \sum_{|I| \leq \delta} S_{l,I} z^I \right)^{j_l} \in k \big[ \{ S_{l,I} \}_{1\leq l \leq s, \vert I \vert \leq \delta } \big]~,
\end{equation}
 one sees that  $(*) = \vert R(\alpha_{n_j})\vert$, and so $\vert R(\alpha_{n_j})\vert \to \vert R(\alpha)\vert$ as $n$ tends to infinity
 since $\alpha_{n_j} \to \alpha$. 
Moreover, it is clear from \eqref{eq: P alpha} that $ R(\al)  = g(\ev(\al)(z) )$, and so the sequence
 $\left| g(f_{n_j}(z))\right| $ converges to $\left| g(\ev(\al)(z) ) \right| $,
 concluding the proof.
\qed

\section{Parametrization of the space of analytic maps}\label{section mor}

We interpret analytic maps between an open and a closed polydisk as rigid points of the spectrum of a suitable Banach $k$-algebra.
Our aim is to build an infinite dimensional analytic space $\mor (\D^r, \bar{\D}^s)$ that parametrizes in a suitable sense the set of all analytic maps 
 from $\D^r$ to $\bar{\D}^s$.
This construction shall be used in the next section to prove Theorem \ref{thm a}.

\smallskip

We shall assume troughout this section that $k$ is algebraically closed. 

We fix two integers $r,s>0$.

\subsection{Construction of the Banach $k$-algebra  $\T_\infty^{r,s}$}

Pick some integer $\delta \in \N^*$. 
 Recall from \S \ref{section polynomial}
that the set of all polynomial maps
 $P: \A^{r,\an} \to \A^{s,\an} $  of degree at most $\delta$
such that $P(\D^r) \subset \bar{\D}^s $
can be endowed with a natural structure of affinoid space whose affinoid algebra  is the Tate algebra $k\{a_{1,I}, \cdots, a_{s,I}\}_{|I| \le \delta}= k\{a_{l,I}\}_{|I| \le \delta, 1 \le l \le s}$.
We  denote this space by $\mor^{r,s}_\delta$.
 It is isomorphic as a $k$-analytic space to the unit polydisk $\bar{\D}^{s(\delta+1)^r}$. 

\medskip

Observe that for any given  $\delta \in \N^*$ there exists a natural truncation map $\mathrm{pr}_\delta : \mor_{\delta+1}^{r,s} \to \mor^{r,s}_\delta $,
which is a 
surjective analytic map dual to the inclusion of Tate algebras $k\{a_{l,I}\}_{|I|\le \delta,  1 \le l \le s} \subset k\{a_{l,I}\}_{|I|\le \delta+1,  1 \le l \le s}$.
These inclusions are isometric and we may so consider the inductive limit of this directed system.
It is a normed $k$-algebra that we  denote by $\T^{r,s}$. 

\smallskip

In order to describe the elements of $\T^{r,s}$ and its norm, we introduce the set
$\mathcal{S}$  of all maps $\M: \{ 1, \ldots , s\} \times \N^r \to \N$ having finite support  and set $|\M| = \sum_{l,I} \M(l,I)$ for every  $\M \in \mathcal{S}$.
We define $\mathcal{S}_\delta$ as the subset of $\mathcal{S}$ consisting of  all $\M \in \mathcal{S}$ such that $\M(l,I) =0$ for all $|I| \ge \delta+1$.
Observe that no such set $\mathcal{S}_\delta$  is finite.
Given  $a= \big( (a_{1,I})_{|I| \le \delta}, \ldots, (a_{s,I})_{|I| \le \delta} \big)$ and $\M \in \mathcal{S}$, we write
$$a^\M = \prod_{ 1 \le l \le s, I \in \N^r } a_{l,I}^{\M(l,I)}.$$
The $k$-algebra $\T^{r,s}$ consists of all power series that are of the form 
$$
\sum_{ \M \in \mathcal{S}_\delta } g_\M \cdot a^\M,
$$
for some  $\delta \in \N$ and whose coefficients $g_\M \in k$ are such that   $|g_\M| \to 0$ as $|\M|\to\infty$. 

Let us describe the norm on $\T^{r,s}$. Observe that by the definition of $\mathcal{S}_\delta$, every element $\sum_{ \M \in \mathcal{S}_\delta } g_\M \cdot a^\M \in \T^{r,s}$ belongs to the Tate algebra $k\{a_{l,I}\}_{|I|\le \delta,  1 \le l \le s} $, and we may associate to  $\sum_{ \M \in \mathcal{S}_\delta } g_\M \cdot a^\M$ the norm on $k\{a_{l,I}\}_{|I|\le \delta,  1 \le l \le s} $.
Since the inclusions of  $k\{a_{l,I}\}_{|I|\le \delta,  1 \le l \le s}$ in $ k\{a_{l,I}\}_{|I|\le \delta+1,  1 \le l \le s}$ are isometric, this norm is well-defined.

\begin{obs}
The $k$-algebra $\T^{r,s}$ is not complete. Take for instance $r=s=1$ and consider the sequence
$f_n = \sum_{i=1}^n g_i \cdot a_i \in \T^{1,1}$. This is a  Cauchy sequence as soon as  the coefficients $g_i \in k$ are such that $|g_i| \to 0$ when $i\to\infty$, but 
it does not have any limit in $\T^{1,1}$. 
\end{obs}

The completion $\T^{r,s}_\infty$ of $\T^{r,s}$ is the Banach $k$-algebra consisting of all power series 
$$
\sum_{ \M \in \mathcal{S}} g_\M \cdot a^\M
$$
such that $|g_\M| $ tends to zero with respect to the filter of cofinite subsets, i.e. such that for all $\epsilon >0$ the set of $\M \in \mathcal{S}$ such that
 $|g_\M| > \epsilon$  is finite. The norm on $\T^{r,s}_\infty$ is the Gauss norm given by $\max_\M |g_\M|$.

\begin{lema}\label{lema function algebra}
The algebra  $\T^{r,s}_\infty$ is a Banach function $k$-algebra.
\end{lema}

\begin{proof}
Recall that a Banach $k$-algebra is a function algebra when its sup norm is equivalent to its given norm, see \cite[3.8.3]{BGR}.
 We shall prove that the sup norm on $\T^{r,s}_\infty$ is actually equal to the Gauss norm. 
 To see this, pick a nonzero element $f = \sum_{ \M \in \mathcal{S}} g_\M \cdot a^\M\in \T^{r,s}_\infty$. The set $\mathcal{G}$ of indices $\M$ such that $|g_\M| > \frac12 |f|$ is finite, 
 so that $\sum_{ \M  \in \mathcal{G}} g_\M \cdot a^\M$ is a polynomial in finitely many variables,
 hence attains its maximum at a point in the unit polydisk by \cite[5.1.4]{BGR}. The lemma follows from \cite[Collorary 3.8.2/2]{BGR}.
\end{proof}

\begin{defini}
The space $\mor (\D^r, \bar{\D}^s) $ is the analytic spectrum of the Banach algebra $\T^{r,s}_\infty$.
\end{defini}

In particular, $ \mor (\D^r, \bar{\D}^s) $  is compact,   because it is the analytic spectrum of the $k$-Banach algebra $\T^{r,s}_\infty$.

\medskip

For every $\delta \in \N$, the isometric inclusion $k\{ a_{l,I}\}_{|I|\le \delta, 1 \le l \le s} \subset \T^{r,s}_\infty$
defines a natural surjective continuous map $\mathrm{Pr}^\infty_\delta : \mor (\D^r, \bar{\D}^s)   \to \mor^{r,s}_\delta$.
We may as well consider the inverse limit of all the spaces $\mor^{r,s}_\delta$, induced by the truncation maps $\mathrm{pr}_\delta : \mor^{r,s}_{\delta+1} \to \mor^{r,s}_{\delta}$.
 These maps verify the equality  $\mathrm{pr}_\delta \circ \mathrm{Pr}^\infty_{\delta+1} =\mathrm{Pr}^\infty_\delta$
and induce a continuous map $\varphi: \mor (\D^r, \bar{\D}^s)  \to \varprojlim_{\delta} \mor^{r,s}_\delta$.

We shall consider the inclusions $i_\delta : \mor^{r,s}_\delta \to \mor (\D^r, \bar{\D}^s) $ given by the bounded morphism $\T^{r,s}_\infty \rightarrow  k \{ a_{l,I} \}_{|I| \le \delta , 1\le l \le s}$, sending $a_{l,I}$ to itself if $|I| \le \delta$ and to 0 otherwise. These are closed immersions.

\begin{prop}\label{mor proj limit}
The map $\varphi: \mor (\D^r, \bar{\D}^s)   \to \varprojlim_{\delta} \mor^{r,s}_\delta $ is a homeomorphism.
\end{prop}

\begin{proof}
The inverse limit $\varprojlim_{\delta} \mor^{r,s}_\delta $ is compact by Tychonoff.

%\smallskip

Let us show that $\varphi: \mor (\D^r, \bar{\D}^s)  \to \varprojlim_{\delta} \mor^{r,s}_\delta$ is bijective. 

Fix  $\delta \ge 0$.
 Let  $\pi_\delta: \varprojlim_{\delta} \mor^{r,s}_\delta \to \mor_\delta$ be the natural map and $\mathrm{pr}_\delta : \mathrm{Mor}^{r,s}_{\delta+1} \to \mor^{r,s}_{\delta}$ the truncation map.
We know that  $\mathrm{Pr}^\infty_\delta = \pi_\delta \circ \varphi : \mor (\D^r, \bar{\D}^s)   \to \mor^{r,s}_\delta$. 
Pick a point $y \in  \varprojlim_{\delta} \mathrm{Mor}^{r,s}_\delta$ and consider $\pi_\delta (y) \in \mor_\delta^{r,s}$. 
Consider the set $K_\delta$ consisting of all the points $\al \in \mor  (\D^r, \bar{\D}^s)  $ such that $\mathrm{Pr}^\infty_\delta ( \al) = \pi_\delta (y)$.
The closed immersion $i_\delta : \mor^{r,s}_\delta \to \mor (\D^r, \bar{\D}^s) $ constructed above is a section of $\mathrm{Pr}^\infty_\delta$.
Thus,  the map $\mathrm{Pr}^\infty_\delta$ is surjective and the subset $K_\delta$ is non-empty.
Clearly, we have that $K_{\delta + 1} \subseteq K_\delta$. 
Every $K_\delta$ is compact and so the intersection  $\cap_{\delta \ge 0} K_\delta$ is nonempty. 
This shows that $\varphi$ is surjective.

For the injectivity, let $\al,\al '$ be two points in $\mor (\D^r,\bar{\D}^s)$ having the same image in $\varprojlim_{\delta} \mathrm{Mor}^{r,s}_\delta$. 
 We have to check that $|g(\al)| = |g(\al')|$ for every $g \in \T^{r,s}_\infty$, that by density reduces to the case where $g \in \T^{r,s}$.
We know that  $\mathrm{Pr}^\infty_\delta ( \al) = \mathrm{Pr}^\infty_\delta (\al') \in \mor_\delta^{r,s}$ for all $\delta$.
  Given $g \in \T^{r,s}$ observe that it lies in $k \{ a_{l,I}\}_{|I| \le \delta, 1 \le l \le s}$ for some $\delta \ge 0$. Thus, 
\begin{equation*}
|g(\al)| = |g(\mathrm{Pr}^\infty_\delta (\al)) | = |g(\mathrm{Pr}^\infty_\delta (\al')) |= |g (\al')| ~,
\end{equation*}
concluding the proof.
\end{proof}

Recall  from \S \ref{section intro berko} the definition of the complete residue field $\mathcal{H}(\al)$ of a point $\al \in \mor (\D^r, \bar{\D}^s) $. We say that $\al$ is rigid when $\mathcal{H}(\al) = k$.
To simplify notation, we write $\al_\delta = \mathrm{Pr}^\infty_\delta (\al)$.

\begin{prop}\label{residue field direct limit}
Let $\al$ be a  point in $ \mor (\D^r, \bar{\D}^s) $.
For every $\delta \in \N$, the inclusion of Banach $k$-algebras $k\{a_{l,I}\}_{1 \le l \le s, |I|\le \delta} \subset  \T^{r,s}_\infty$  induces an extension of valued fields 
$\mathcal{H}(\al)/\mathcal{H}(\al_\delta)$. 

The complete residue field $\mathcal{H}(\al)$ is isomorphic to the 
completion of the inductive limit of valued fields $\varinjlim_\delta \mathcal{H}(\al_\delta)$.
\end{prop}

\begin{proof}
A point $\al \in \mor (\D^r, \bar{\D}^s) $ corresponds to a seminorm on the $k$-algebra $\T^{r,s}_\infty$, whose restriction to $k\{ a_{l,I} \}_{|I| \le \delta,  1\le l \le s}$ is the seminorm $\al_\delta$. 
The kernel of $\al_\delta$ is the intersection of  $k\{ a_{l,I} \}_{|I| \le \delta ,  1\le l \le s}$ with $\ker(\al)$.
This induces inclusions 
\begin{equation}\label{eq H a}
k\{ a_{l,I} \}_{|I| \le \delta ,  1\le l \le s} / \ker (\al_\delta) \subset \T^{r,s}_\infty / \ker (\al).
\end{equation}

It follows that there are inclusions $\mathcal{H}(\al_\delta) \subset \mathcal{H}(\al)$, and thus the direct limit of the $\mathcal{H}(\al_\delta)$ is naturally contained in $\mathcal{H}(\al)$.
In order to show that $\mathcal{H}(\al)$ is isometrically isomorphic to the completion of $\varinjlim_\delta\mathcal{H}(\al_\delta) $, it suffices to show that $\varinjlim_\delta\mathcal{H}(\al_\delta)$ is dense in $\mathcal{H}(\al)$.

Consider the field $K := \varinjlim_\delta \mathrm{Frac} \left( k\{ a_{l,I} \}_{|I| \le \delta ,  1\le l \le s} / \ker (\al_\delta) \right) $.
It is clear that  $K$ is contained in  $\varinjlim_\delta \mathcal{H}(\al_\delta)$. By \eqref{eq H a} and by the definition of $\T^{r,s}_\infty$, we also know that $K$ is dense in $\mathrm{Frac} \left( \T^{r,s}_\infty/ \ker(\al) \right) $.
The latter is by definition dense in $\mathcal{H}(\al)$, which proves that  $\varinjlim_\delta\mathcal{H}(\al_\delta)$ is dense in $\mathcal{H}(\al)$.
\end{proof}

\begin{prop}\label{prop rigid dense}
The set of rigid points in $\mor (\D^r, \bar{\D}^s)$  is dense.
\end{prop}

\begin{proof}
Pick any point $\al \in \mor (\D^r, \bar{\D}^s)$. 
For every $\delta \in \N$, pick a sequence of rigid points $\al_n^{(\delta)} \in \mor_\delta^{r,s}$ converging to $\al_\delta$. 
By Proposition \ref{mor proj limit} and Proposition \ref{residue field direct limit}, a point in $ \mor (\D^r, \bar{\D}^s)$ is rigid if and only if for every $\delta \in \N$ its projection to  $\mor_\delta^{r,s}$ is rigid. 

We may view each point $\al_n^{(\delta)}$ as a rigid point in $\mor (\D^r, \bar{\D}^s) $
via de map $i_\delta : \mor^{r,s}_\delta \to \mor (\D^r, \bar{\D}^s) $ constructed above.
We claim that $\al$ lies in the closure of  the set $\{ \al_n^{(\delta)} \}_{n, \delta}$. 
Indeed, fix an open neighbourhood $U$ of $\al$. It is a finite intersection of open sets of the form $\{ \beta \in  \mor (\D^r, \bar{\D}^s) :  | g (\beta)| - |g(\al) | | \le r\}$ for some $r \le 1$.
Since $\T^{r,s}$ is dense in $\T_\infty^{r,s}$, we may assume that $g \in \T^{r,s}$. 
Thus, for sufficiently large $\delta$ one has that $|g(\al_\delta) | = |g (\al)|$. 
Moreover, by construction we have $|g( \al_n^{(\delta)} )| \stackrel{n \to \infty}{\rightarrow} |g(\al_\delta) | $.
It follows that for $\delta, n \gg 0$, the points $\al_n^{(\delta)} $ belong to $U$.
\end{proof}

\subsection{Universal property of the space $\mor (\D^r, \bar{\D}^s)$}

Let us specify  in which sense $\mor (\D^r, \bar{\D}^s)$  parametrizes the space of analytic maps from $\D^r_k$ to $\bar{\D}^s_k$.
Recall from \S \ref{section maps} that a  morphism  between  the spectra of two  Banach $k$-algebras  is by definition a continuous map induced by a bounded morphism between the underlying algebras.
In  the same fashion, an analytic map from a good $k$-analytic space $W$ into $\mor  (\D^r, \bar{\D}^s)$ is given by an affinoid covering $\{ W_i \}$ of $W$ and analytic maps $W_i \to \mor  (\D^r, \bar{\D}^s)$, which are induced by bounded morphisms of Banach $k$-algebras and are compatible with the restrictions.

\begin{theorem}\label{thm universal property of mor}
There exists an analytic map 
$\Phi : \mor  (\D^r, \bar{\D}^s) \times \D^r \to \bar{\D}^s$
satisfying the following universal property.
Let $W$ be the analytic spectrum of a Banach function $k$-algebra or any good reduced 
$k$-analytic space.
Then, for any analytic map
$F: W \times \D^r \to \bar{\D}^s$ there exists a unique morphism
$G: W \to \mor   (\D^r, \bar{\D}^s)$ such that 
$F(x,z) = \Phi( G(x), z)$ for all $x\in W(k)$ and $z \in \D^r(k)$.
\end{theorem}

Notice that this property uniquely characterizes the space $ \mor  (\D^r, \bar{\D}^s)$, as it is the analytic spectrum of a  Banach function algebra.

\begin{proof}
Let us first construct the analytic map $\Phi : \mor  (\D^r, \bar{\D}^s) \times \D^r \to \bar{\D}^s $.
The assignment
$$(S_1, \ldots, S_s) \mapsto \left(\sum_{I \in \N^r} a_{1,I} T^I, \ldots , \sum_{I \in \N^r} a_{s,I} T^I \right)$$
 defines a bounded morphism of Banach $k$-algebras
 $\psi : k\{S_1, \ldots, S_s\} \to \mathcal{T}_\infty^{r,s} \{\rho^{-1}T_1, \ldots, \rho^{-1}T_r\}$ for every positive $\rho<1$,
 and thus an analytic map 
$\Phi : \mor  (\D^r, \bar{\D}^s) \times \D^r \to \bar{\D}^s$.

\medskip

We now prove the universal property.
Suppose first that $W$ is the analytic spectrum of a  Banach function $k$-algebra $\mathcal{A}$. 
In particular, we may assume that $\mathcal{A}$ is endowed with the sup norm.
Recall that the norm on the complete tensor product $\mathcal{A} \hat{\otimes}_k \mathcal{T}_r$  agrees with the Gauss norm $\| \sum_I b_I T^I \| = \sup_I |b_I|$.
Let $F: W\times \D^r \to \bar{\D}^s$ be an analytic map, 
 induced by some bounded homomorphism of Banach $k$-algebras
\begin{equation*}
(S_1, \ldots, S_s) \mapsto \left(\sum_{I \in \N^r} b_{1,I} T^I, \ldots , \sum_{I \in \N^r} b_{s,I} T^I\right) ~ ,
 \end{equation*}
 where $b_{l,I} \in \mathcal{A}$ are such that $\sup_{l,I} |b_{l,I} (x)| \le 1$   for all $x \in W$ by Lemma \ref{lema function algebra}.

 Consider the analytic map $G: W \to \mor (\D^r, \bar{\D}^s)$  given by $a_{l,I} \mapsto b_{l,I}$ for all $I \in \N^r$ and all $1 \le l \le s$. 
A rigid point $x \in W$ together with a rigid point $z \in \D^r$  defines a rigid point in the product $W \times \D^r$, and by construction we have $F(x,z) = \Phi( G(x), z)$.
 
 Conversely, let $H: W \to \mor  (\D^r, \bar{\D}^s)$ be an analytic map
  sending $a_{l,I}$ to some $c_{l,I} \in \mathcal{A}$ and satisfying $F(x,z) = \Phi( H(x), z)$ for all $x \in W(k)$ and all $z \in \D^r(k)$.
 For every fixed $x \in W(k)$, consider the analytic map $z\in \D^r \mapsto \Phi (H(x),z)$. 
 By hypothesis, it agrees with the map $z\in \D^r \mapsto \Phi (G(x),z)$, and so $b_{l,I}(x) = c_{l,I}(x)$ for every $I \in \N^r$ and $1 \le l \le s$. As the equalities hold for every rigid $x \in W$, we conclude that  $H=G$.
 \medskip
 
 Let now $W$ be an arbitrary  good $k$-analytic space.
Let $\{ W_i \}$ be an affinoid covering of $W$ inducing an analytic map $F : W \times \D^r \to \bar{\D}^s$. 
  By the previous case, for every affinoid domain $W_i$ of $W$ there exists a unique  analytic map $G_i : W_i \to \mor  (\D^r, \bar{\D}^s)$, induced by a bounded morphism of Banach algebras, such that $F(x,z) = \Phi( G_i(x), z)$ for all $x\in W_i(k)$ and $z \in \D^r(k)$.
 By construction, the maps $G_i$ agree on the intersections $W_i \cap W_j$ and are compatible with the restrictions. 
\end{proof}

\subsection{Points of $\mor  (\D^r,\bar{\D}^s)$ as continuous maps $\D^r \to \bar{\D}^s$}\label{section continuous maps}

The following theorem specifies in which sense the points of the space $\mor  (\D^r,\bar{\D}^s)$ correspond to continuous maps from $\D^r$ to $\bar{\D}^s$.

\begin{theorem}\label{thm mor}
There exists a map $\mathrm{Ev}$ from  $\mor  (\D^r, \bar{\D}^s)$ to the space of continuous functions $\mathcal{C}^0(\D^r, \bar{\D}^s)$
such that the following holds:
\begin{enumerate}\renewcommand{\labelenumi}{\roman{enumi})}
\item
The map $\ev(\al)$ is analytic if and only if  $\al \in \mor  (\D^r, \bar{\D}^s)$ is rigid. In that case, the map $\ev(\al)$ is precisely  $\Phi(\al,\cdot)$.
\item
For any fixed $z\in \D^r$, the assignment $\al \in \mor (\D^r, \bar{\D}^s) \mapsto \ev(\al)(z)$ is a continuous map.
\end{enumerate}
\end{theorem}

\begin{proof}
The map $\ev: \mor (\D^r, \bar{\D}^s) \to \mathcal{C}^0(\D^r, \bar{\D}^s)$ is  given as follows.
Fix a point $\al \in \mor (\D^r, \bar{\D}^s)$ and consider the first projection $\pi_1: \mor  (\D^r, \bar{\D}^s) \times \D^r \to \mor  (\D^r, \bar{\D}^s)$.
 The fibre $\pi_1^{-1} (\al)$ is canonically isomorphic to $\D^r_{\mathcal{H}(\al)}$ (cf. \S \ref{section maps}).
We can thus consider the inclusion map $\iota_{\mathcal{H}(\al)}: \D^r_{\mathcal{H}(\al)} \to \mor  (\D^r, \bar{\D}^s) \times \D^r$, given by
\begin{eqnarray}\label{eq inclusion fibre}
\T^{r,Y}_\infty \{ \rho^{-1} T_1, \ldots, \rho^{-1} T_r \}& \to & \mathcal{H}(\al) \{ \rho^{-1} T_1, \ldots, \rho^{-1} T_r \} \nonumber \\
T_i & \mapsto & T_i\\
a \in  \T^{r,s}_\infty  & \mapsto & \chi_\al ( a) \nonumber
\end{eqnarray}
for $\rho <1$, where $\chi_\al: \T^{r,s}_\infty \to \mathcal{H}(\al)$ denotes the character associated to the point $\al$.
Let $\sigma_{\mathcal{H}(\al)/k}: \D^r \to \D^r_{\mathcal{H}(\al)}$ be the continuous map discussed in \S \ref{subsection universal}. Let $\Phi : \mor  (\D^r, \bar{\D}^s) \times \D^r \to \bar{\D}^s$ be the analytic map from Theorem \ref{thm universal property of mor}.
We set: 
$$\ev (\al) = \Phi \circ \iota_{\mathcal{H}(\al)} \circ \sigma_{\mathcal{H}(\al)/k}~.$$ 
Clearly, $\ev (\al)$ is a continuous map from $\D^r$ to $\bar{\D}^s$.
Specifically, for any $z \in \D^r$ and for any $g= \sum_{J \in \N^s} g_J S^ J$ in $k\{S_1, \ldots , S_s \}$, we have
\begin{equation}\label{eq EV(a) explicit}
\left| g( \ev(\al) (z) ) \right| =  \left| \sum_J g_J \prod_{l=1}^s \left( \sum_I \chi_\al ( a_{l,I}) \cdot T^I \right)^{j_l}   (\sigma_{\mathcal{H}(\al)/k}(z) ) \right| .
\end{equation}

Pick a rigid  point $\al \in \mor (\D^r, \bar{\D}^s)$, i.e. such that  $\mathcal{H}(\al) = k$.
 In this situation,  the fibre $\pi_1^{-1}(\al)$ is homeomorphic to $\D^r$, and so $\iota_{\mathcal{H}(\al)}$ is in fact an analytic map between $k$-analytic spaces, and 
 the map  $\sigma_{\mathcal{H}(\al)/k}$ is the identity on $\D^r$. 
 Then, for every $z \in \D^r$ the pair $(\al,z)$ defines a point in $\mor (\D^r, \bar{\D}^s) \times \D^r$, and so
 $\iota_k(z) = (\al,z)$. Thus,  $\ev(\al) = \Phi(\al, \cdot)$ is analytic.

Suppose conversely that  $\ev (\al)$ is analytic.
 %Then  $\chi_\al (a_{l,I})$ does not belong to $k$ for some $I \in \N^r$ and for some $1\le l \le s$.
It follows from \eqref{eq EV(a) explicit} that the map $\ev(\al)$  
can be decomposed as $\ev (\al) = \pi_{K/k} \circ F \circ \sigma_{K/k}$, where $F: \D^r \to \bar{\D}^s$ is the $K$-analytic map 
$$F(z) = \left(\sum_{I \in \N^r} \chi_\al ( a_{1,I}) \cdot z^I , \ldots, \sum_{I \in \N^r} \chi_\al ( a_{s,I}) \cdot z^I \right).$$
 It suffices to treat the case $s=1$.
Since $\ev(\al)$ is analytic, we may find coefficients $b_I \in k$ bounded by $1$ such that $\ev (\al) (z) =   \sum_{I \in \N^r} b_I z^I$  for every $z \in \D^r (k)$. 
Notice that the equality $\pi_{K/k} \left( \sum_{I \in \N^r} \chi_\al ( a_I) \cdot z^I \right)= \sum_{I \in \N^r} b_I z^I \in k$ implies that $\sum_{I \in \N^r} \chi_\al ( a_I) \cdot z^I  \in k$, as $k$ is algebraically closed.

Suppose by contradiction that $\al$ is not a rigid point and consider the equation 
\begin{equation}\label{eq al analytic}
 \sum_{I \in \N^r} b_I z^I = \sum_{I \in \N^r} \chi_\al ( a_I) \cdot z^I ~,
\end{equation}
where we may assume that  every $\chi_\al (a_I)$ is either $0$ or does not belong to  $k$. Since $\al$ is not rigid, not all of them are zero.
  We may consider the nonempty set $M \subseteq \N^r$ consisting of all the multi-indices $I \in \N^r$ such that $\chi_\al ( a_{I}) \notin k$.
Let $\mathrm{P}$ be the Newton polytope of $M$, i.e. the
convex hull of the union of all upper-quadrants $I + \R_+^r$ 
with $I\in M$. It is a non-compact polytope in $\R_+^r$ whose extremal points all belong to $M$.
 
Pick any extremal point $p$ of $\mathrm{P}$, and take any hyperplane in $\R^r$ with integer coefficients $H = \{  \beta_1 x_1 + \ldots + \beta_r x_r = \beta_0 \}$
  intersecting the polytope $\mathrm{P}$ exactly at the point $p$.
 In other words, we have %$p_1 \beta_1 + \ldots + p_r \beta_r = \beta_0$ and
 \begin{equation}\label{eq hyperplan}
  i_1  \beta_1 + \ldots + i_r \beta_r > \beta_0
 \end{equation}
  for every $I \in \N^r$ distinct from $p$ intervening in \eqref{eq al analytic}.
Fix any  $\lambda \in k$ with $|\lambda | < 1$ and consider the rigid point $z = (\lambda^{\beta_1}, \ldots , \lambda^{\beta_r}) \in \D^r$. Then,
\begin{eqnarray*}
 \sum_{I \in \N^r} b_I z^I & = &  \sum_{I \in \N^r} b_I (\lambda^{\beta_1}, \ldots , \lambda^{\beta_r})^I =
    \sum_{I \in \N^r} b_I \lambda^{i_1  \beta_1 + \ldots + i_r \beta_r} =\\
    & =& b_p \lambda^{\beta_0} +  \sum_{I \in \N^r, I \neq p} b_I  \lambda^{i_1  \beta_1 + \ldots + i_r \beta_r}
  =   b_p \lambda^{\beta_0} + O (\lambda^{\beta_0})~,
\end{eqnarray*}
where the last equality follows from \eqref{eq hyperplan}.
It follows that $b_p \lambda^{\beta_0} + O (\lambda^{\beta_0}) = \chi_\al (a_p ) \lambda^{\beta_0} +  O (\lambda^{\beta_0})$,
and hence  $\chi_\al (a_p ) = b_p \in k$. 
Repeating this procedure at every vertex of the polytope $\mathrm{P}$, we conclude that $\chi_\al (a_I) \in k$ for every $I \in \mathrm{P}$, contradicting the fact that  $\al$ is not rigid.

 \medskip

Let us now prove the continuity statement.
Fix a point $z \in \D^r$.
It suffices to check that for any sequence of points $ \{ \al_n  \}  \subset \mor  (\D^r, \bar{\D}^s)$ converging to some $\al \in \mor  (\D^r, \bar{\D}^s)$, we have
$\ev(\al_n)(z) \to \ev(\al)(z)$.

Consider the  second projection 
$\pi_2: \mor  (\D^r, \bar{\D}^s) \times \D^r \to \D^r$. The fibre $\pi_2^{-1} (z)$ is isomorphic to $\mor  (\D^r, \bar{\D}^s)_{\mathcal{H}(z)}$. 
The inclusion  map $\iota_{\mathcal{H}(z)}: \mor  (\D^r, \bar{\D}^s)_{\mathcal{H}(z)} \to \mor  (\D^r, \bar{\D}^s) \times \D^r$ is given by
\begin{eqnarray*}
\T^{r,s}_\infty  \{ \rho^{-1} T_1, \ldots, \rho^{-1} T_r \}& \to & \T^{r,s}_\infty \hat{\otimes}_k  \mathcal{H}(z)\\
T_i & \mapsto & \chi_z(T_i)\\
a_{l,I} & \mapsto &  a_{l,I}
\end{eqnarray*}
for $\rho <1$, where $\chi_z: k \{ \rho^{-1} T_1, \ldots, \rho^{-1} T_r\} \to \mathcal{H}(z)$ denotes the character associated to the point $z$. 
Pick some converging power series $g = \sum_{I \in \N^r } g_I T^I$ in $\T^{r,s}_\infty  \{ \rho^{-1} T_1, \ldots, \rho^{-1} T_r \}$ and compute:
\begin{multline}\label{eq: P z}
 \left| g \left( \iota_{\mathcal{H}(z)} \circ \sigma_{\mathcal{H}(z)/k} (\al)\right) \right|  
= \big| \big( \sum_{I \in \N^r} g_I \cdot  \chi_z(T)^I \big) \left( \sigma_{\mathcal{H}(z)/k} (\al)\right) \big|  \\
=  \max_{I \in \N^r } \left| g_I(\al) \right| \cdot \left| \chi_z(T)^I  \right|_{\mathcal{H}(z)}
 =  \max_{I \in \N^r } \left| \chi_\al(g_I) \right|_{\mathcal{H}(\al)} \cdot \left| T^I (z)  \right|  \\
 =   \big| \big( \sum_{I \in \N^r } \chi_\al(g_I) \cdot T^I \big) \left( \sigma_{\mathcal{H}(\al)/k} (z)\right) \big|  
 =  \left| g \left( \iota_{\mathcal{H}(\al)} \circ \sigma_{\mathcal{H}(\al)/k} (z)\right) \right|  ~.
\end{multline}
That is,  for all fixed $z \in \D^r$ and $\al \in \mor  (\D^r, \bar{\D})$, we have $$\iota_{\mathcal{H}(\al)} \circ \sigma_{\mathcal{H}(\al)/k} (z) = \iota_{\mathcal{H}(z)} \circ \sigma_{\mathcal{H}(z)/k} (\al)~.$$

Consider the continuous map $\Psi(z): \mor  (\D^r, \bar{\D}^s) \to \bar{\D}^s$, defined as the composition 
$\Psi(z) = \Phi \circ \iota_{\mathcal{H}(z)} \circ \sigma_{\mathcal{H}(z)/k}$. For every fixed $\al \in \mor  (\D^r, \bar{\D}^s)$ and every fixed $z \in \D^r$, we have
$$\Psi (z) (\al) = \ev(\al) (z).$$
If $\al_n$ is a sequence of points in $\mor  (\D^r, \bar{\D}^s)$ converging to $\al$, then the continuity of $\Psi(z)$ implies that $\Psi(z)(\al_n)$ converges to $\Psi(z)(\al)$ as $n$ goes to infinity, and so
$$\ev(\al^{(n)}) (z) \stackrel{n\to \infty}{\longrightarrow} \ev(\al) (z)~,$$
concluding the proof.
\end{proof}

\subsection{The space $ \mor  (\D^r, \bar{\D}^s)$  is  Fréchet-Urysohn.}

We prove a technical result that is a key step in the proof of Theorem \ref{thm a}.

\begin{theorem}\label{mor is FU}
The space $ \mor  (\D^r, \bar{\D}^s)$  is  Fréchet-Urysohn. 
\end{theorem}

We follow Poineau's  proof of the fact that  analytic spaces are Fréchet-Urysohn \cite[Proposition 5.2]{Poineau},  which in turn relies on \cite[Théorème 4.22]{Poineau}.  

Recall that a subset $\Gamma$ of the analytic spectrum of a $k$-Banach algebra $(\mathcal{A}, \| . \| )$ is a boundary if   for every $g \in \mathcal{A}$ there exists some $x \in \Gamma$  such that $| g(x) | = \| g \|$.
A closed boundary is called the \emph{Shilov boundary} if it is the smallest closed subset $\Gamma$ of $\mathcal{M(A)}$ satisfying this property.
Since we have excluded the trivially valued case and the norm on $\T^{r,s}_\infty $ is multiplicative, there exists a Shilov boundary in $\mor (\D^r, \bar{\D}^s)$ by \cite[Theorem C]{Shilov}.
 
In the following we deal with subfields $l$ of $k$ that are of \emph{countable type} over the prime subfield $k_p$ of $k$, i.e. such that $l$  has a dense $k_p$-vector subspace of countable dimension.
 
\medskip

The following proposition is an infinite dimensional analogue of \cite[Théorème 4.22]{Poineau}.

\begin{prop}\label{prop bord de shilov}
 For every point $\al$ in $ \mor  (\D^r, \bar{\D}^s)$ there exists a subfield $l$ of $k$ that is of countable type over the prime subfield $k_p$ of $k$ and satisfying the following property.  Let $l^\prime$ be any  subfield of $k$ with $l \subset l' \subset k$ and let $\pi^\infty_{k/l'}:  \mor  (\D^r, \bar{\D}^s) \to  \mor  (\D^r, \bar{\D}^s)_{l'}$ be the base change morphism.
 Then $\al$ is the unique point in the Shilov boundary of the fibre $(\pi^\infty_{k/l'})^{-1}(\pi^\infty_{k/l'} (\al))$.
\end{prop}

\begin{proof}
The space $\mor (\D^r, \bar{\D}^s)$ is the projective limit of $\mor^{r,s}_{\delta}$ with the morphisms $\mathrm{Pr}^\infty_{\delta, k}: \mor (\D^r, \bar{\D}^s) \to \mor^{r,s}_{\delta}$ for $\delta \in \N^*$ (cf. Proposition \ref{mor proj limit}).
 A point $\al$ in $\mor (\D^r, \bar{\D}^s)$ is thus determined
  by a sequence $(\al_\delta)_{\delta \ge 0}$, where each $\al_\delta$ lies  in $\mor^{r,s}_{\delta}$ and  satisfies $\mathrm{pr}_{\delta + 1} (\al_{\delta + 1}) = \al_\delta$ for the projections $\mathrm{pr}_{\delta + 1}: \mor^{r,s}_{\delta + 1} \to  \mor^{r,s}_{\delta}$.
 
  To every $\al_\delta$ we apply  \cite[Th\'eor\`eme 4.22]{Poineau}. We  obtain a field $l^\delta \subset k$ that is of countable type over the prime subfield  $k_p$ of $k$ and such that for any
 subfield $l^\delta \subset l' \subset k$ the point $\al_\delta$ is the only point in the Shilov boundary of $(\pi^\delta_{k/l'})^ {-1}(\pi^\delta_{k/l'} (\al_\delta))$, where $\pi^\delta_{k/l'}: \mor^{r,s}_{\delta} \to \mor^{r,s}_{\delta, l^\prime}$ denotes the base change morphism.

Let $l$ be the subfield of $k$ generated by all the $l^\delta$. 
By construction, $l$  is of countable type over $k_p$. We may assume in addition that $l$ is algebraically closed.

The equality $\pi^\delta_{k/l'} \circ ~\mathrm{Pr}^\infty_{\delta, k} = \mathrm{Pr}^\infty_{\delta, l^\prime} \circ ~ \pi^\infty_{k/l'}$
implies that  $\mathrm{Pr}^\infty_{\delta, k}$ maps the fibre $(\pi^\infty_{k/l'})^{-1}(\pi^\infty_{k/l'}(\al))$ to the fibre $(\pi^\delta_{k/l'})^{-1}(\pi^\delta_{k/l'}(\al_\delta))$.
   We show that $\al$ belongs to the Shilov boundary of $(\pi^\infty_{k/l'})^{-1}(\pi^\infty_{k/l'}(\al))$.
 Pick an element $g \in \T^{r,s}_\infty $.
 As $\T^{r,s}$ is dense in $\T^{r,s}_\infty$, we may assume that $g$ lies in $ k\{a_{l,I} \}_{|I| \le \delta, 1\le l \le s }$ for some  $\delta\ge 0$. 
 Thus, 
$|g(\al)| = |g (\al_\delta) |$, which is the maximum value of $g$, since $\al_\delta$ belongs to the Shilov boundary of $(\pi^\delta_{k/l'})^{-1}(\pi^\delta _{k/l'} (\al_\delta))$.

 Pick a point $\beta \in (\pi^\infty_{k/l'})^{-1}(\pi^\infty_{k/l'}(\al))$ different from $\al$, i.e. such that $\beta_\delta \neq \al_\delta$ for some $\delta \ge 0$. 
 As $\al_\delta$ is the unique point in the Shilov boundary of $(\pi^\delta_{k/l'})^{-1}(\pi^\delta _{k/l'} (\al_\delta))$, we may find some 
  $g \in k\{a_{l,I} \}_{|I| \le \delta }$ such that
\begin{equation*}
|g(\beta)| = |g (\beta_\delta)|< |g (\al_\delta) | = |g(\al)|,
\end{equation*}
showing that $\al$ is the unique point in  the Shilov boundary of the space $(\pi^\infty_{k/l'})^{-1}(\pi^\infty_{k/l'} (\al))$.
\end{proof}

\begin{proof}[Proof of Theorem \ref{mor is FU}]
Let $A$ be any subset of $\mor (\D^r, \bar{\D}^s)$ and let $\alpha$ be a point in the closure of $A$. Let $l$ be the subfield of $k$ associated to $\al$ from Proposition \ref{prop bord de shilov}.
Let $l\subset l' \subset k$ be any subfield of $k$ that is of countable type over $l$. 
Every polydisk $\mor_{\delta, l'}^{r,s}$ is first countable,  and as a consequence so is the countable product of all the $\mor_{\delta, l'}^{r,s}$. The space $\mor (\D^r, \bar{\D}^s)_{l'}$ is a subspace of the product $\prod_{\delta} \mor_{\delta, l'}^{r,s}$  by Proposition \ref{mor proj limit}, and thus is first countable.

Copying Poineau's proof of \cite[Proposition 5.2]{Poineau} and using Proposition \ref{prop bord de shilov}, we may  find a sequence of points $\al_n$ in $A$ converging to $\al$.
\end{proof}

\section{Montel's theorem}\label{section preuve montel}

This section is devoted to the proof of Theorem \ref{thm a}. 
We first apply the results and constructions from the previous sections to prove the case where the base field $k$ is algebraically closed and next we generalize this argument to an arbitrary non-Archimedean complete field.

\subsection{Proof of Theorem  ~\ref{thm a} in the algebraically closed case}

Let  $k$ be an algebraically closed complete non-Archimedean field.

Let $X$ be a good, reduced, $\sigma$-compact $k$-analytic space without boundary and $Y$ a strictly $k$-affinoid space.
Pick a sequence of analytic maps $f_n: X \to Y$. 
We claim that there exists a subsequence that is pointwise converging to a continuous map.

\medskip 

Since any $k$-affinoid space $Y$ may be embedded into a polydisk, we may readily assume that $Y= \bar{\D}^s$ for some integer $s$.

Assume first that $X = \D^r$. In this case, each analytic map $f_n$ corresponds to a rigid point $\al_n$ in  $\mor  (\D^r, \bar{\D}^s)$ by Theorem \ref{thm universal property of mor}. 
Since the space $\mor  (\D^r, \bar{\D}^s)$ is compact and Fr\'echet-Urysohn by Theorem \ref{mor is FU}, we may find a converging subsequence $\al_{n_j}$ converging to some point $\al \in  \mor  (\D^r, \bar{\D}^s)$.
The continuous map $\ev(\al) : \D^r \to \bar{\D}^s$ is the limit map of the subsequence $f_{n_j}$ by Theorem \ref{thm mor}.

Suppose now that $X$ is a basic tube in the sense of \S \ref{section tubes}.
Let $\hat{X} $ be a strictly $k$-affinoid space and $\hat{X} \to \bar{\D}^r$ a distinguished closed immersion  such that $X$ is isomorphic to $\hat{X} \cap \D^r$ (cf. Proposition \ref{basic tube intersection}). 
We may thus write $X$ as a growing countable union of affinoid spaces $X = \bigcup_{0 <\rho < 1} X_\rho$. Moreover, since $k$ is algebraically closed, we may take every $\rho$ in $|k^\times|$.
As the affinoid algebra corresponding to $\hat{X}$ is isomorphic to the quotient of the Tate algebra $\T_r$ by some closed ideal $I$, we may assume that the affinoid algebra $\mathcal{A}_\rho$ of each $X_\rho$ is of the form $k \{\rho^{-1} T_1 , \ldots , \rho^{-1} T_r \}$ modulo the ideal generated by $I$.
In particular, we have distinguished closed immersions $\varphi_\rho: X_\rho \to \bar{\D}^r(\rho)$.

Let $f_n : X \to \bar{\D}^s$ be a sequence of analytic maps. 
Fix $0 < \rho < 1$, $\rho \in |k^\times|$. We may apply Proposition \ref{prop distinguished} to the restriction of 
 $f_n$ to $X_\rho$ to obtain an  analytic map  $g_n^{(\rho)}: \bar{\D}^r(\rho) \to \bar{\D}^s$ extending $f_n |_{X_\rho}$.
Indeed, $f_n |_{X_\rho}$ is given by elements  $f_1^{(n)} , \ldots , f_s^{(n)} \in \mathcal{A}_\rho$ of norm at most $1$.
As we have a distinguished epimorphism $k \{\rho^{-1} T_1 , \ldots , \rho^{-1} T_r \} \to \mathcal{A}_\rho$, we may lift each $f_l^{(n)}$, $l= 1, \ldots , s$, to an element in $k \{\rho^{-1} T_1 , \ldots , \rho^{-1} T_r \}$ having the same norm. These define  analytic maps $g_n^{(\rho)}: \bar{\D}^r(\rho) \to \bar{\D}^s$ satisfying $g_n^{(\rho)} \circ \varphi_\rho = f_n |_{X_\rho}$ for all $n \in \N$.
We now apply the previous case to the restricted sequence $\{ g_n^{(\rho)}|_{\D^r(\rho)}\}_n$.
We conclude by  a diagonal extraction argument.

Consider now $X$ as in the theorem. 
Being $\sigma$-compact, $X$ is the union of  countably many compact sets $K_n$. Since it is a good analytic space without boundary, 
each compact set $K_n$ is included in a finite union of open sets, each isomorphic to a basic tube. 
It follows that $X$ is a countable union of basic tubes $U_m$. 
By the previous case, on every open set $U_m$ there exists a subsequence  converging pointwise, and extracting diagonally we may find a subsequence $\{f_{n_j} \}$ converging pointwise on the whole  $X$.
The limit is continuous on every $U_m$ and hence on $X$ since they are open.
\qed

\subsection{Proof of Theorem \ref{thm a} over an arbitrary base field}\label{sec:other base field}

Let $K$ be a completed algebraic closure of $k$, and $X_K, Y_K$ be the scalar extensions of $X$ and $Y$  respectively, see \S\ref{subsection universal}.

Pick a sequence $f_n: X \to Y$ of analytic maps and consider the analytic maps $F_n: X_K \to Y_K$ induced by the base change.
The following diagram commutes:
\begin{equation*}
\xymatrix{
   X_K \ar[r]^{F_n} \ar[d]_{\pi_{K/k}} & Y_K \ar[d]^{\pi_{K/k}} \\
   X \ar[r]^{f_n} & Y
}
\end{equation*}
 Observe that the  analytic space $X_K$ is good and $\sigma$-compact, since the preimage $\pi_{K/k}^{-1}(U)$ of an affinoid domain $U$ of $X$ is an affinoid domain  in $X_K$.
It follows directly from the definition of the interior that  $X_K$ is boundaryless (\cite[Proposition 3.1.3]{Berk}).
 Thus, by the algebraically closed case of Theorem \ref{thm a} proved above, we may assume that $F_n$ is pointwise converging to a continuous map $F: X_K \to Y_K$. 
 Pick a point $z \in X$. 
 As $\pi_{K/k}$ is surjective, we may choose  a point $z' \in \pi_{K/k} ^{-1}(z) $. It follows that
  $f_n ( z ) = f_n(\pi_{K/k} (z') ) = \pi_{K/k} \circ F_n(z')$, which tends to $\pi_{K/k} \circ F(z') = : f(z)$ as $n$ goes to infinity. The limit map $f$ is well-defined. Indeed, if $z'$, $z''$ are two points in $\pi_{K/k}^{-1}(z)$, then
  \begin{eqnarray*}
  && \lim_n \pi_{K/k} \circ F_n(z') = \lim_n f_n (\pi_{K/k} (z') )=\\ & = & \lim_n f_n (\pi_{K/k} (z'') ) = \lim_n \pi_{K/k} \circ F_n(z'') ~.
 \end{eqnarray*}

It remains to check that $f$ is continuous. Let $A$ be any closed (hence compact) subset of $Y$. By continuity, the set $F^{-1} \left( \pi^{-1}_{K/k} (A)\right)$ is closed. Recall that the map $\pi_{K/k}: X_K\to X$ is proper. Since  $X_K$ and $X$ are locally compact, then $\pi_{K/k}$ is closed. As a consequence, $f^{-1}(A) = \pi_{K/k} \left( F^{-1} \circ \pi^{-1}_{K/k} (A) \right)$ is closed.
\qed

\subsection{Fields with countable residue field}

We observe in this section that part of the assertion of Theorem \ref{thm a} extends to maps between any $k$-affinoid spaces when the residue field of $k$ is countable. 
Specifically, we do not exclude source spaces with boundary and show that one may always extract an everywhere converging subsequence.
This section will not be used in the rest of the paper, since the limits we obtain  this way are not necessarily  continuous.

\medskip

Recall that the boundary of an affinoid space can  be written as a finite union of affinoid spaces defined over some extension of $k$, see \cite[Lemma 3.1]{Ducrospolytopes}. 
Here we shall only use the following observation. Consider the closed $N$-dimensional polydisk $\bar{\D}^N$, and denote by $p_i: \bar{\D}^N \to \bar{\D}$  the projection to the $i$-th coordinate.
 Recall that  the boundary of $\bar{\D}$ consists only of the Gauss point.
  It follows from Lemma \ref{interieur - points fermes} that $p_i^{-1}(x_g)$ is contained in the boundary of $\bar{\D}^N$ for every $ i=1, \cdots,  N$.  
Let now $x$ be a point in $\partial \bar{\D}^N$ and consider the commutative diagram:
\begin{equation*}
 \xymatrix{
\bar{\D}^N \ar[d]_{\mathrm{red}} \ar[r]^{p_i} & \bar{\D} \ar[d]^{\mathrm{red}}\\
\A^N_{\tilde{k}}  \ar[r]^{\tilde{p_i}} &\A^1_{\tilde{k}}\\
}
\end{equation*} 
Suppose that $p_i(x) \neq x_g$ for all $i$. 
By Lemma \ref{interieur - points fermes}, the point $\widetilde{p_i (x)}$ is  closed in $\A^1_{\tilde{k}}$ corresponding to some maximal ideal $\langle T_i - \zeta_i \rangle \subset \tilde{k}[T_i]$ for every $i= 1, \ldots, N$.
The commutativity of the diagram implies that the maximal  ideal 
 $\langle T_1 - \zeta_1, \ldots,   T_N- \zeta_N \rangle$ of $\tilde{k}[T_1, \ldots, T_N]$
 is contained in the prime ideal  corresponding to $\red(x)$. As a consequence,   $\red (x) \in \A^N_{\tilde{k}}$ is  closed, contradicting the fact that $x$ belongs to $\partial \bar{\D}^N$.

The boundary of  $\bar{\D}^N$ is thus equal to the union $p_1^{-1}(x_g) \cup \ldots \cup p_N^{-1}(x_g).$
Observe that  each fibre $p_i^{-1}(x_g)$ is isomorphic to $\bar{\D}^{N-1}_{\mathcal{H}(x_g)}$.

\begin{prop}\label{corollary Cp}
Suppose $k$ is a  non-Archimedean complete valued field that is algebraically closed and such that $\tilde{k}$ is countable. Let $X$ and $Y$ be $k$-affinoid spaces and assume that $X$ is reduced and distinguished. Then, every sequence of analytic maps $f_n: X \to Y$ has an everywhere pointwise converging subsequence.
\end{prop}

\begin{proof}
We may assume  $X= \bar{\D}^r$, $Y = \bar{\D}^s$ as in the proof of Theorem \ref{thm a}. 
The set of connected components of the interior of $\bar{\D}^r$ is in bijection with the set of $\tilde{k}$-points 
on its reduction $\A^r_{\tilde{k}}$ and hence is countable.
 
We now argue inductively on $r$.  When $r=1$, then the boundary of $\bar{\D}$ consists of a single point, namely the Gauss point.  
We may therefore apply Theorem \ref{thm a} to each of the (countably many) components
of the interior of $\bar{\D}$ and apply a diagonal extraction argument to conclude.

\smallskip

 Assume now that the statement holds for the polydisk of dimension $r-1$ defined over \emph{any} complete valued field with  countable residue field, and pick a sequence of analytic maps $f_n : \bar{\D}^r \to \bar{\D}^s$.
As before,  we apply  Theorem \ref{thm a} to each of the (countably many) components
of the interior of $\bar{\D}^r$ so that we may suppose that $f_n$ converges pointwise on the interior of $\bar{\D}^r$.

The boundary of $\bar{\D}^r$ is the union of $r$ unit polydisks of dimension $r-1$ defined over the field $\mathcal{H}(x_g)$ by our previous discussion. On each of  these  we may apply the induction hypothesis, as the field $\widetilde{\mathcal{H}(x_g) }$ is isomorphic to $\tilde{k}(T)$, which is countable.
This concludes the proof.
\end{proof}

\subsection[Analytic properties of pointwise limits]{Analytic properties of pointwise limits of analytic maps}\label{section thm b}

Continuous maps of the form  $\mathrm{Ev}(\al): \D^r \to \bar{\D}^s$ are very special, 
as they exhibit properties that are distinctive of analytic maps. 
We shall prove that they lift to analytic maps after a suitable base change and that the graph of $\mathrm{Ev}(\al)$ is well-defined in the analytic product $ \D^r \times \bar{\D}^s$ and not just in the topological product $|\D^r | \times | \bar{\D}^s| $.

\smallskip

Recall  from \S \ref{subsection universal} the definition of the continuous map $\sigma_{K/k}: X \to X_K$.

\begin{theorem} \label{thm lifting}
Let $k$ be a complete non-Archimedean field that is algebraically closed.

Let $\al$ be a point in $\mor ( \D^r, \bar{\D}^s)$.
 Then there exists a closed subset $\Gamma_\al$ of  $\D^r \times \bar{\D}^s $ such that the first projection $\pi_1: \Gamma_\al \to \D^r$ is a homeomorphism
  and such that for  every $z \in \D^r$ the image of $\Gamma_\al \cap \pi_1^{-1}(z)$ under
  the second projection is the point  $\ev(\al)(z) \in \bar{\D}^s$. 
  
 Moreover, there exist a complete extension $K$ of $k$ and a $K$-analytic map $F_\al: \D^r_K \to \bar{\D}^s_K$ such that $\ev(\al) = \pi_{K/k} \circ F_\al \circ \sigma_{K/k}$.
\end{theorem}

\begin{proof}
Fix a point $\al \in \mor (\D^r, \bar{\D}^s)$ and denote by $\mathcal{H}(\al)$ its complete residue field. We define $\Gamma_\al$ as the image of a continuous map $\psi :\D^r \to \D^r \times \bar{\D}^s$, that we construct as follows.

Let $\iota_{\mathcal{H}(\al)} : \D^r_{\mathcal{H}(\al)} \to \mor  (\D^r, \bar{\D}^s) \times \D^r$ be the inclusion map defined in \eqref{eq inclusion fibre}. 
Let $\Upsilon: \mor   (\D^r, \bar{\D}^s)  \times \D^r \to \D^r \times \bar{\D}^s$ be the analytic map induced by
\begin{eqnarray*}
k\{ \rho^{-1} T_1, \ldots, \rho^{-1} T_r \} \{S_1, \ldots, S_s \}& \to & \T^{r,s}_\infty \{ \rho^{-1} T_1, \ldots, \rho^{-1} T_r \} \nonumber \\
T_i & \mapsto & T_i\\
S_l & \mapsto & \sum_I a_{l,I} T^I. \nonumber
\end{eqnarray*}
Let $\sigma_{\mathcal{H}(\al)/k}: \D^r \to \D^r_{\mathcal{H}(\al)}$. 
 We set $\psi= \Upsilon \circ \iota_{\mathcal{H}(\al)} \circ \sigma_{\mathcal{H}(\al)}$. 
 Explicitly, $\psi$ is induced by the analytic map $\Psi: \D^r \to \D^r \times \bar{\D}^s$ that maps  any $z \in \D^r$  to the seminorm sending every $g \in \T_s \{ \rho^{-1} T_1, \ldots, \rho^{-1} T_r \}$, which is of the form $g = \sum_{J \in \N^s} g_J S^J$ with  $g_J \in k\{ \rho^{-1} T_1, \ldots, \rho^{-1} T_r \}$  are such that $|g_J| \to 0$ as $|J| \to 0$, to the following real number:
 
 \begin{equation}\label{eq graph}
 \left| g(\Psi(z))\right| = \left| \sum_J g_J \prod_{l=1}^s \big( \sum_I \chi_\al ( a_{l,I}) \cdot T^I \big)^{j_l}   (\sigma_{\mathcal{H}(\al)/k}(z) ) \right| .
 \end{equation}
 Consider the projections $\pi_1$ and $\pi_2$ on $\D^r \times \bar{\D}^s$ to the first and second factor respectively.
It is an immediate consequence of the previous computation and \eqref{eq EV(a) explicit} that  
$$\pi_2 (\psi(z)) = \ev (\al) (z).$$

 If no variables $S_l$ appear in the expression of $g \in \T_s \{ \rho^{-1} T_1, \ldots, \rho^{-1} T_r \}$, 
then $g$ lies in  the algebra $k\{ \rho^{-1} T_1, \ldots, \rho^{-1} T_r \}$. Thus, by \eqref{eq graph} we see that
$|g( \Psi(z) )|= |g(z)|$, and so
 $$\pi_1 (\psi(z)) = z~.$$
 
 It remains to check that the image $\Gamma_\al$ of $\psi$ is a closed subset of $\D^r \times \bar{\D}^s$. Let $z_n$ be a sequence of points in $\D^r$ such that $\psi(z_n)$ converges to some point $x$ in $\D^r \times \bar{\D}^s$. 
 As $\pi_1 (\psi(z_n)) = z_n$, we see that $z_n$ converges to $\pi_1 (x) \in \D^r$, and by continuity of $\psi$ we have that $x = \psi(\pi_1 (x))$ lies in $\Gamma_\al$.
The set $\Gamma_\al$ is so sequentially closed, and hence closed.
 
\smallskip

Consider now the continuous map $\ev(\al) : \D^r \to \bar{\D}^s$.
Let $K$ be the complete residue field $\mathcal{H}(\al)$.
Consider the $\mathcal{H}(\al)$-analytic map
$$F_\al = \left( \sum_{I \in \N^r} \chi_\al ( a_{1,I})\cdot T^I , \ldots , \sum_{I \in \N^r} \chi_\al ( a_{s,I})\cdot T^I\right)~.$$
A direct computation together with \eqref{eq EV(a) explicit} shows that $\ev(\al)= \pi_{\mathcal{H}(\al)/ k} \circ F_\al \circ \sigma_{\mathcal{H}(\al)/k}$.
\end{proof}

\subsection{Proof of Theorem \ref{thm b}}
Let $Y$ be any $k$-affinoid space. We may fix a closed immersion of $Y$ into some polydisk $\bar{\D}^s$ and assume $Y= \bar{\D}^s$.

Suppose first that $X =\D^r$.
Each analytic map $f_n$ is of the form $f_n = \ev(\al_n)$ for some rigid point $\al_n \in \mor  (\D^r, \bar{\D}^s)$ by Theorem \ref{thm mor}. It was shown in Proposition \ref{mor proj limit} 
that the space $\mor  (\D^r, \bar{\D}^s)$  is Fr\'echet-Urysohn so that  we may assume that 
$\al_n$ converges to some point $\al \in \mor  (\D^r, \bar{\D}^s)$.
The limit map $f$ is precisely $\ev(\al)$ (cf. Theorem \ref{thm mor}) 
and we conclude by Theorem \ref{thm lifting}.

Let now $X$ be any good, boundaryless, reduced $k$-analytic space. Pick a point $x\in X$ and an affinoid neigbourhood $Z$ of $x$ containing $x$ in its interior.
Fix a distinguished closed immersion of $Z$ into some closed unit polydisk $\bar{\D}^r$.
For every $n$ we may find an analytic map $\hat{f}_n : \bar{\D}^r \to \bar{\D}^s$ such that $\hat{f}_n|_Z = f_n$ by Proposition  \ref{prop distinguished}.
We now apply   the previous case to the restriction of $\hat{f}_n$ to $\D^r$, concluding the proof.
\qed

\section{Weakly analytic maps}\label{section weakly analytic}

In this section we look more precisely at the properties of continuous limits
of analytic functions, as obtained in Theorem \ref{thm b}.

As before, $k$ is any complete non trivially valued non-Archimedean  field which is algebraically closed.

\subsection{Definition and first properties}
We begin with a definition.

\begin{defini} 
Let $X$ and $Y$ be any two good $k$-analytic spaces, 
and  let $f: X \to Y$ be a continuous map. 

We say that $f$  is weakly analytic if for every point  $x \in X$ there exist an affinoid neighbourhood $U$ of $x$, a complete 
field extension $K/k$ and an analytic map $F:U_K \to Y_K$ such that $f_{|_U} = \pi_{K/k} \circ F \circ \sigma_{K/k}$.
\end{defini}

It will be convenient to denote by $\mathrm{WA}(X, Y)$ the set of all weakly analytic maps from $X$ to $Y$.

\smallskip

Clearly, the set $\mor_k (X,Y)$ of analytic maps from $X$ to $Y$ is  a  subset of $\mathrm{WA}(X, Y)$. It is also a strict subset if $Y$ has dimension at least $1$,
since any constant map is weakly analytic, but it is analytic only if the constant is a rigid point.

\begin{prop}\label{equivalence faiblement analytique}
Let $X$ be a basic tube and $Y$ be a $k$-affinoid space.
Let $f:X \to Y$ be a continuous map. The following two conditions are equivalent.
\begin{enumerate}\renewcommand{\labelenumi}{\roman{enumi})}
\item  For any point $x \in X$ there exist an affinoid neighbourhood $Z$ of $x$ and a sequence of analytic maps $f_n: Z \to Y$ pointwise converging to $f|_Z$.
\item 
For any point $x \in X$ there exist an affinoid neighbourhood $Z$ of $x$,  a complete extension $K$ of $k$ and an analytic map $F:Z_K \to Y_K$ such that $f|_Z = \pi_{K/k} \circ F \circ \sigma_{K/k}$.
\end{enumerate}
\end{prop}

A consequence of the previous result is that when $X$ has no boundary, a continuous map $f: X\to Y$ is weakly analytic whenever 
for every point  $x \in X$ there exists a basic tube $U$ containing  $x$ and a sequence of analytic maps $f_n$ from $U$ to $Y$ that converge pointwise to $f$.

\begin{proof}
The implication i) $\Rightarrow$ ii) is precisely Theorem \ref{thm b}, since basic tubes are boundaryless.

Suppose that ii) is satisfied. 
Choosing a closed immersion $Y \to \bar{\D}^s$, we may assume $Y= \bar{\D}^s$. 
Pick a point $x \in X$ and an affinoid neighbourhood $Z$ of $x$ such that there exists a complete extension $K/k$ and a $K$-analytic map  $F: Z_K \to \bar{\D}^s_K$  such that $f|_Z = \pi_{K/k} \circ F \circ \sigma_{K/k}$. 
By Proposition \ref{prop distinguished}, we may find an analytic map $\hat{F}: \D^r_K \to \bar{\D}^s_K$ that agrees with $F$ on $Z_K \cap \D^r_K$. 
By Theorem \ref{thm mor}, there exists a rigid point  $a \in \mor (\D^r, \bar{\D}^s)_K$ such that $\hat{F}=\Phi(a,\cdot)$.
The point $\al = \pi_{K/k}^\infty (a)$ in $\mor (\D^r, \bar{\D}^s)$ is not rigid in general, but we may find points $\al_n \in \mor (\D^r, \bar{\D}^s) (k)$ converging to $\al$ by Proposition \ref{prop rigid dense}, since $k$ is assumed to be non trivially valued.
The analytic maps $\ev( \al_n)$ converge pointwise to $\ev (\al): \D^r_k \to \bar{\D}^s_k$ by Theorem \ref{thm mor}, and by construction we have that $\ev(\al) = \pi_{K/k} \circ \hat{F} \circ \sigma_{K/k}$ , see Theorem \ref{thm lifting}.
\end{proof}

\subsection{Rigidity of weakly analytic maps}

We prove here the following statement:

\begin{prop}
Suppose $f: X \to Y$ is a weakly analytic map, where $Y$ is a curve. 
If  $x$  is a rigid point that is mapped to a non-rigid point by $f$, then  $f$ is locally constant near  $x$.
\end{prop}

\begin{proof} 
Let $x\in X$ be a rigid point such that $y=f(x)$ is not rigid. 
Since this is a local statement, we may replace $X$ and $Y$ by affinoid neighbourhoods of $x$ and  $y$ respectively. In particular, we may assume that  $X=\bar{\D}^r$ and $x=0$. After maybe reducing $X$, there exists an extension $K$ of $k$ and a $K$-analytic map $F: X_K \to Y_K$ such that $f = \pi_{K/k} \circ F \circ \sigma_{K/k}$. Observe that $F(x)$ is a rigid point of $Y_K$.

Suppose first that $Y= \bar{\D}$. The fact that $y$ is not rigid  means that $y$ has positive diameter, i.e.  
$$ \inf_{a \in k^\circ} |(T-a)(y)| = r >0.$$
 By continuity, we can find a polyradius $\epsilon > 0$ such that every rigid point $z$ in $\D^r_K(0; \epsilon)$ satisfies $|F(z)-F(0)|_K < r$, where $|.|_K$ denotes the absolute value on $K$. Pick a point $a \in k^\circ$. For every rigid point $z\in\D^r_K(0; \epsilon)$, we get
\begin{eqnarray*}
|(T-a)(y)| &=& \max \left\{ |F(z)-F(0)|_K, |(T - a) (y)| \right\}\\
& = & \max \left\{ |F(z)-F(0)|_K, |(T - a) (\pi_{K/k} \circ F(0))| \right\}\\ 
&=& \max \left\{ |F(z)-F(0)|_K, |F(0) - a|_K \right\}\\
& = & | F (z) - a |_K  \\
&= &  |(T-a)(\pi_{K/k} \circ F (z) )|~.
\end{eqnarray*}
Thus, $F$ maps the polydisk $\D^r_K(0; \epsilon)$ into the fibre $\pi_{K/k}^{-1}(y)$.
As $$\sigma_{K/k} (\D^r_k (0;\epsilon) ) \subseteq \D^r_K (0;\epsilon),$$ we conclude that $f$ is locally constant near $0$.

\smallskip

For $Y$ any affinoid of dimension 1 there exists a finite morphism $\varphi: Y \to \bar{\D}$  by Noether's Lemma. By what precedes, the composition $\varphi \circ f$ is locally constant near $0$, and by finiteness so is $f$.
\end{proof}

\begin{example}
The previous result does not hold if $Y$ has dimension greater than 2. Consider for instance the weakly analytic map $f: \bar{\D} \to \bar{\D}^2$ given by $f= \pi_{K/k} \circ F \circ \sigma_{K/k}$, where $K= \mathcal{H}(x_g)$ and $F(z) = (x_g, z)$. No rigid point in $\bar{\D}$ has rigid image under $f$, but $f$ is not locally constant at these points.
\end{example}

\subsection{Weakly analytic maps from curves}

\begin{prop} 
Let  $f: X\to Y$ be a weakly analytic map, where $X$ is a curve. 
If there exists a converging sequence of rigid points of $X$ whose images under $f$ are rigid points, then $f$ is analytic.
\end{prop}

\begin{obs}
Let $X$ be a $k$-affinoid space.
Let $f: X \to \bar{\D}^s$ be a continuous map such that 
there exists a complete extension $K/k$  such that $f = \pi_{K/k} \circ F \circ \sigma_{K/k}$ for some $K$-analytic map $F$.
We may assume that the extension $K/k$ is of countable type  \cite[\S 2.7]{BGR}.

Indeed, let $\mathcal{A}$ be the underlying $k$-affinoid algebra of $X$ and fix an epimorphism $k \{ r^{-1} T \} \to \mathcal{A}$ such that $\mathcal{A}$ is isomorphic as a Banach algebra to $k \{ r^{-1} T \} /I$ for some closed ideal $I \subset k \{ r^{-1} T \}$.
Extending scalars, we see that $\mathcal{A}_K$ is isomorphic to $K \{ r^{-1} T \}  /I$ as  a $K$-affinoid algebra.
 The map $F$ is then determined by  elements $F_1, \cdots , F_s \in \mathcal{A}_K$ with $|F_l|_{\sup} \le 1$, and hence the expression of $F$ contains at most countably many elements of $K$.
\end{obs}

\begin{proof}
Pick any sequence $x_n \in X(k)$  such that $f(x_n)$ are also rigid, and assume that $\lim_n x_n  = x$. Here $x$ may be non-rigid.
We may replace $X$ by some affinoid neighbourhood of $x$ and assume that $f = \pi_{K/k} \circ F \circ \sigma_{K/k}$ for some complete extension $K/k$ and some $K$-analytic map $F$. Observe that  $f(x_n) = F(x_n) \in Y(k)$. 
We may as well replace $Y$ by  an affinoid neighbourhood of $f(x)$ and embed it in some polydisk $\bar{\D}^s$.

\smallskip

Let $\mathcal{A}$ be the underlying $k$-affinoid algebra of $X$. The map $F$ is then determined by  elements $F_1, \cdots , F_s$ in the $K$-affinoid algebra $\mathcal{A}_K$   with $|F_l|_{\sup} \le 1$.
Pick any real number $\alpha >1$. 
By \cite[Proposition 2.7.2/3]{BGR}  there is an $\alpha$-cartesian Schauder basis $\{ e_j \}_{j\in \N}$ of $K$,
and we may  choose $e_0=1$ by \cite[Proposition 2.6.2./3]{BGR}. 

\smallskip
Fix an epimorphism $\T_M \to \mathcal{A}_K$ and lift every $F_l$ to an element $G_l$ in $\T_M$.
Then for every $l=1, \cdots, s$ we can develop $G_l= \sum_I a_I^l T^I$ with $a_I^l\in K$ and such that $|a_I^l|_K \to 0$ as $|I|$ goes to infinity.
 Using the Schauder basis we may find elements $ a_{I,j}^l\in k$ such that $a_I^l = \sum_j  a_{I,j}^l e_j$ and satisfying 
 \begin{equation*}
|a_{I,j}^l|_k \le \max_j |a_{I,j}^l|_k \le \alpha |a_I^l|_K ~.
\end{equation*}
 Since $ \alpha |a_I^l|_K \to 0$ as as $|I|$ goes to infinity, the series $A_l^j = \sum_I  a_{I,j}^l T^I$  defines an element in $\T_M$.
 Thus, we obtain a converging power series  $G_l = \sum_j \big( \sum_I  a_{I,j}^l T^I \big) e_j$.
 Recall that $F_l(x_n) \in k$ for all $n$, and so $G_l(x_n) \in k$. We infer  that for $j \ge 1$ and for all $n$, $A_l^j(x_n) = 0$. Each of these $A_l^j$ defines in turn an analytic map  on $X$ that vanishes at every $x_n$, and hence is constant equal to zero on $X$ by the principle of isolated zeros. 
It follows that   $F_{l|_X}=A_l^0$  for every $1\le i \le s$, thus they are defined over $k$.
\end{proof}

We observe that the previous result does not hold in higher dimension.

\begin{example}
Let  $\zeta_n \in k$, $|\zeta_n|=1$, $|\zeta_n -\zeta_m| = 1$ for $n\neq m$.
Let $f$ be the weakly analytic map obtained as the limit of the sequence $f_n: \D^2 \to \bar{\D}^1$, given  on the rigid points by $f_n(z_1, z_2) =  \zeta_n z_1  + z_2$.
The map $f$ is not analytic, since the rigid point $(\lambda, 0) \in \D^2$, $0<|\lambda | <1$,  is mapped to the point in $\bar{\D}$ corresponding to the closed ball $\bar{B}(0; |\lambda| )$.
However, the set $\{ 0 \} \times \D^1(k)$  is mapped  to the set of rigid points.
\end{example}

A consequence of the previous result is the following statement that can be viewed as
the principle of isolated zeroes for weakly analytic maps.

\begin{prop}\label{isolated zeros}
Let $f: X\to Y$ be  a non constant weakly analytic map where $X$ is a curve without boundary. 
Then the fibre of any rigid point in $Y$ contains no accumulation point.
\end{prop}

\begin{proof}
Let $y \in Y(k)$  and suppose there exist points  $x_n \in X$ converging to a point $x$ and such that $f(x_n) = y$ for all $n$.
In this situation, we may assume $Y=\bar{\D}^s$, $y=(0, \cdots, 0)$ and replace $X$ with some affinoid neighbourhood of $x$ such that $f$ lifts to a $K$-analytic map $F$ over some complete extension $K/k$. This map $F$ is given by some  elements $F_1, \cdots, F_s$ in the underlying affinoid algebra of $X_K$ of norm at most $1$. 

 The point $y$ is rigid and so it has only one preimage under $\pi_{K/k}$. Thus, 
\begin{equation*}
 (0, \cdots, 0) = f(x_n) = F \circ \sigma_{K/k} (x_n) \in \bar{\D}^s_K
\end{equation*} 
for all $n$.  Since $X$ is a curve and $F$ is non-constant (otherwise $f$ would be so), $F^{-1}(0)$ is included in the 
set of rigid points of $X$. It follows that every $\sigma_{K/k}(x_n)$ is rigid.
Each component $F_l$ of  $F$ defines an analytic map between the curves $X_K$ and $\bar{\D}_K$ and admits a sequence of zeros with an accumulation point $\sigma_{K/k}(x)$. It follows that every $F_l$  is identically zero, hence so is $f$. 
\end{proof}

\subsection{A conjecture on weakly analytic maps}

On basic tubes, we conjecture that weakly analytic maps can be globally lifted to analytic maps.

\begin{conj}\label{relever sur un ouvert basique}
Let $Y$ be a $k$-affinoid space and $X$ a basic tube. Let  $f:X \to Y$ be a weakly analytic map. Then, there exist a complete extension $K/k$ and $F: X_K \to Y_K$ analytic such $f= \pi_{K/k} \circ F \circ \sigma_{K/k}$.
\end{conj}

Notice that a weakly analytic map can be locally lifted to an analytic map over some complete extension of $k$. Conjecture \ref{relever sur un ouvert basique} means that this can be done globally.

\begin{obs}
In the case when $X$ and $Y$ are polydisks, Conjecture \ref{relever sur un ouvert basique}  amounts to saying that the map $\ev$ is surjective onto the set $\mathrm{WA}(X,Y)$.

\end{obs}

 The map $\ev$ becomes closed by Theorem \ref{thm mor} for the topology of the pointwise convergence, and so $\mathrm{WA}(X,Y)$ becomes Fréchet-Urysohn for this topology.
 
\medskip

Observe that if Conjecture \ref{relever sur un ouvert basique} holds, then using Theorem \ref{mor is FU} we have:

\begin{theorem}\label{montel for weakly analytic}
Suppose that Conjecture \ref{relever sur un ouvert basique} holds.

Let $X$ be a boundaryless $\sigma$-compact $k$-analytic space and $Y$ a $k$-affinoid space. 
Then, every sequence of weakly analytic maps $f_n: X \to Y$ admits a subsequence that is pointwise converging
to a weakly analytic map $f:X \to Y$. 
\end{theorem}

As a consequence, we have:

\begin{cor}
Suppose that Conjecture \ref{relever sur un ouvert basique} holds.
Let $X$ be a boundaryless $\sigma$-compact $k$-analytic space and $Y$ a $k$-affinoid space. Let $\{f_n\} \subset \mathrm{WA}(X, Y)$ be a  sequence converging to some continuous map $f$. Then, $f$ is weakly analytic. 
\end{cor}

\section{Applications to dynamics}\label{section dynamics}

In this section, we attach  two different notions of Fatou sets
to an endomorphism $f$ of the projective space $\mathbb{P}^{N,\an}$ of degree at least $2$ and study their geometry, which exhibit similar properties  to the complex case.

\subsection{Strongly pluriharmonic functions}

We recall the definition from \cite{CLheights}:

\begin{defini}
Let $X$ be any boundaryless $k$-analytic space.
 A continuous function $u: X \to \R$ is strongly pluriharmonic if for every $x\in X$ there exist an open neighbourhood $U$ of $x$,   a sequence of invertible analytic functions $h_n$ on $U$ and real numbers $b_n$ such that 
\begin{equation*}
u = \lim_{n \to + \infty}  b_n \cdot \log |h_n|
\end{equation*}
locally uniformly on $U$.
\end{defini}

Harmonic functions have  been widely studied in dimension 1. 
 Baker-Rumely \cite{BR} and Favre-Rivera Letelier \cite{FRergodique}, and  Thuillier  \cite{Thuillier} have defined non-Archimedean analogues of the Laplacian operator,
  on $\mathbb{P}^{1, \an}$ and on general analytic curves respectively.

 If $X$ is an analytic curve, strongly harmonic functions
  are harmonic in the sense of Thuillier. It is not known yet whether the converse holds, see \cite[Remark 2.4.6]{CLheights}.
  However, if $X$ is a connected open subset of $\mathbb{P}^{1, \an}$, then all definitions  agree  by \cite[Corollary 7.32]{BR}.

Observe that over $\C$, pluriharmonic functions are in fact locally  the logarithm of the norm of an invertible function, whereas this is  not true in the non-Archimedean setting. Counterexamples appear already for curves, see \cite[\S 2.3]{CLheights}.

\begin{obs}
Let $X$ be any boundaryless $k$-analytic space.
The set of all strongly pluriharmonic functions on $X$ forms a $\R$-vector space.
\end{obs}

\subsection{Harmonic functions on open subsets of $\mathbb{P}^{1, \an}$}\label{section harmonic}

Recall from \cite[\S 4.2]{Berk} that the analytic projective line $\mathbb{P}^{1, \an}$ is the one-point compactification of $\A^{1, \an}$.
The points in $\A^{1, \an}$  can be explicitly described as follows \cite[\S 1.4.4]{Berk}.

Pick $a\in k$ and $r \in \R_+$ and denote by 
$\bar{B}(a;r)$  the closed ball in $k$ centered at $a$ and of radius $r$.
To  $\bar{B}(a;r)$ we can associate a point $\eta_{a,r} \in \A^{1,\an}$
by setting $|P(\eta_{a,r})| := \sup_{|y-a| \le r} |P(y)|$ for every polynomial $P\in k[T]$. 
Points of the form $\eta_{a,0}$ are called type I points, and these are precisely the rigid points of $\A^{1, \an}$.
Consider the point $\eta_{a,r}$ with $r>0$.  If $r \in |k^\times|$ we say that  $\eta_{a,r}$ is of type II and if $r \notin  |k^\times|$ of type III.
A decreasing sequence of closed balls $\bar{B}(a_i; r_i)$ in $k$ with empty intersection defines a converging sequence of points $\eta_{a_i,r_i} \in \A^{1, \an}$. The limit point  is   called a type IV point.
Any point in  $\A^{1, \an}$ is of one of these four types. 
%%%%%%

\medskip

It is a fundamental fact that the Berkovich projective line carries  a tree structure.
Roughly speaking, it is obtained by patching together one-dimen\-sional line segments
in such a way that it contains no loop. We refer to \cite[\S 2]{Jonssonlow} for a precise definition. Suffice it to say that for any two points $x, y \in \mathbb{P}^{1, \an}$ there exists a closed subset $[x,y] \subset  \mathbb{P}^{1, \an}$ containing $x$ and $y$ that can be endowed with a partial order making it isomorphic to the real closed unit interval $[0,1]$ or to $\{ 0 \}$.
These ordered sets are required to satisfy a suitable set of axioms.
For instance, for any triple $x,y, z$ there exists a unique point $w$ such 
that $[z,x] \cap [y,x] = [w,x] $ and
 $ [z,y] \cap [x,y] = [w,y]$. Any subset of the form $[x, y]$ is called a segment.

As a consequence, $\mathbb{P}^{1, \an}$ is uniquely path-connected, meaning that given any two distinct points $x, y \in \mathbb{P}^{1, \an}$ the image of every injective continuous map $\gamma$ from the real unit interval $[0,1]$ into $\mathbb{P}^{1, \an}$ with $\gamma (0) = x$ and $\gamma(1) = y$ is isomorphic to the segment $[x,y]$.

\smallskip
 
A nonempty closed subset $\Gamma \subseteq \mathbb{P}^{1, \an}$ is called a subtree if it is  connected. 
An endpoint of  $\Gamma$ is a  point $x \in \Gamma$  such that $\Gamma \setminus \{ x \}$ either remains connected or is empty.
For every subtree $\Gamma$  of  $\mathbb{P}^{1, \an}$ there is a canonical retraction $r_\Gamma: \mathbb{P}^{1, \an} \to \Gamma$, which sends a point $x \in \mathbb{P}^{1, \an}$ to the unique point in $\Gamma$ such that the intersection of the segment $[x, r_\Gamma (x)]$ with $\Gamma$ consists only of the point $r_\Gamma (x)$.

A strict finite subtree $\Gamma$ of $\mathbb{P}^{1, \an}$ is the convex finitely many type II points $x_1, \ldots , x_n$. As a set, it is  the union of all the paths $[x_i, x_j]$, $i=1, \ldots, n$.

\medskip

Recall that a disk in $\mathbb{P}^{1,\an}$ is by definition either a  disk in $\A^{1,\an}$ or the complement of a disk in $\A^{1,\an}$.
Basic tubes  in $\mathbb{P}^{1,\an}$ are  \emph{strict simple domains} in the terminology of \cite{BR}.
 They are either  $\mathbb{P}^{1, \an}$ or  strict open disks in $\mathbb{P}^{1,\an}$ with a finite number of strict closed disks of $\mathbb{P}^{1,\an}$ removed. 
In particular, basic tubes different from $\mathbb{P}^{1,\an}$ and strict open disks
can be obtained as an inverse image $r_\Gamma^{-1}(\Gamma^0)$, where $\Gamma$ is a strict finite subtree of $\mathbb{P}^{1, \an}$ and  $\Gamma^0$  the open subset of $\Gamma$ consisting of $\Gamma$ with its endpoints removed. 

\smallskip

Similarly, every connected affinoid domain of $\mathbb{P}^{1,\an}$ is either a closed disk or a closed disk in $\mathbb{P}^{1,\an}$ with a finite number of open disks of $\mathbb{P}^{1,\an}$ removed.
In particular, an affinoid subset of the form $\bar{\D}(a;r) \setminus \bigcup_{i=1}^n \D(a_i; r_i)$ 
is homeomorphic to the Laurent domain  
of underlying affinoid algebra 
$$ k \{ r^{-1}(T-a), r_1 S_1, \ldots, r_n S_n\} / (S_1 (T-a_1)- 1 , \ldots, S_n (T-a_n)- 1).$$

\medskip

Given a subset $W \subset \mathbb{P}^{1, \an}$, denote by $\overline{W}$ its closure and by $\partial_{\mathrm{top}} W$ its topological boundary.
 If $W$ is a basic tube strictly  contained in $\mathbb{P}^{1, \an}$, then $\partial_{\mathrm{top}} W$ consists of a  finite set of type II points.

\begin{prop}\label{prop tubes BR}
Let $U$ be a proper connected open subset of $\mathbb{P}^{1, \an}$. Then there exist an increasing  sequence $W_m$ of basic tubes of $\mathbb{P}^{1, \an}$ exhausting $U$  and a sequence of strictly affinoid subspaces  $X_m$ of $\mathbb{P}^{1, \an}$ satisfying
$$\overline{W}_m \subset X_m \subset W_{m+1} \subset U$$
for every $m\in \N^*$.
\end{prop}

The proof makes extensive use of the tree structure of $\mathbb{P}^{1,\an}$.
Recall from \cite[Appendix B]{BR} that the tangent space at a point $x \in \mathbb{P}^{1,\an}$  is defined as  the set $T_x \mathbb{P}^{1, \an} $ of paths leaving from $x$ modulo the relation   having a common initial segment.
The space $T_x \mathbb{P}^{1, \an} $  is in bijection with the connected components of $\mathbb{P}^{1, \an} \setminus \{ x \}$. 
Given any tangent direction $\vec{v} \in T_x \mathbb{P}^{1, \an} $, we denote by $U(\vec{v})$ the corresponding connected component of $\mathbb{P}^{1,\an} \setminus \{ x\}$. 

\begin{proof}
By \cite[Corollary 7.11]{BR} there exists a sequence of basic tubes $W_m$ exhausting $U$
and such that $\overline{W}_m \subset W_{m+1} \subset U$ for every $m \in \N^*$. 

Fix a positive integer $m>0$. As we have assumed that $U$ is strictly contained in $\mathbb{P}^{1, \an}$, the topological boundary of  $W_m$ is a non-empty finite set of type II points of $\mathbb{P}^{1, \an}$. The convex hull $\Gamma_m$  of $\partial_{\mathrm{top}} W_m$ is thus a subgraph of $\mathbb{P}^{1, \an}$ with finitely many endpoints. 

If $W_m$ is an open disk, we set $X_m$ to be the closed disk of same centre and same radius as $W_m$. 
Otherwise, consider the following strict finite subtree $\Gamma$  of $\mathbb{P}^{1, \an} $. 
Let $\Gamma_m^0$ be the open subset of $\Gamma_m$ consisting of $\Gamma_m$ with its endpoints removed. 
Pick a point $x$ in $\Gamma_m \setminus \Gamma_m^0$. There are at most finitely many tangent directions at $x$  containing points of the complement in $U$ and not contained in $\Gamma_m$.
For every such tangent direction, attach a segment to $\Gamma_m$ in that direction and in such a way that it is contained in $W_{m+1}$ and such that its endpoint is a type II point.
If no such tangent direction exists, lengthen that edge ending at $x$ such that the new endpoint is again of type II and belongs to $W_{m+1}$. Denote by $\Gamma$ the strict finite subtree obtained this way. Observe that all the boundary points of $\Gamma_m$ are contained in  $\Gamma^0$.

Let  $r_{\Gamma}: \mathbb{P}^{1, \an} \to \Gamma$ be the  natural retraction map.
The basic tube $W_m$ is precisely $r_{\Gamma}^{-1} (\Gamma_m^0)$. 
Setting $X_m = r_{\Gamma}^{-1} (\Gamma_m)$, clearly one has $\overline{W}_m \subset X_m \subset W_{m+1}$.
Let $x_{i_1}, \ldots, x_{i_m}$ be the endpoints of  $\Gamma_m$, where $x_{i_j} = \eta_{a_{i_j},r_{i_j}}$ are of type II.
 The set $X_m$ is homeomorphic to $\mathbb{P}^{1,\an}$  minus the strict open disks $\D(a_{i_j}; r_{i_j})$, $j=1, \ldots , m$, and is thus strictly affinoid.
\end{proof}

The following proposition will be essential  for the proof of Theorem \ref{thm c}.

\begin{prop}\label{lema bound harmonic functions}
Let $U$ be a basic tube in  $\mathbb{P}^{1, \an}$.
There exists a positive constant $C$ depending only on $U$ such that for every harmonic function $g: U \to \R$  
 there exists an analytic function $h: U \to \A^{1, \an} \setminus \{ 0 \}$  such that 
$$\sup_U \big| g - \log |h| \big| \le C.$$
\end{prop}

\begin{proof}
If $U$ is either $\mathbb{P}^{1, \an}$ or $\D$, the assertion is trivial, because every harmonic function on $\D$ or on $\mathbb{P}^{1, \an}$ is constant by \cite[Proposition 7.12]{BR}.
We may thus assume that $U$ is 
of the form $\D \setminus \cup_{i=1}^m \bar{\D}(a_i, r_i)$ with $r_i \in |k^\times|$, $0 <r_i < 1$ and $|a_i| <1$ for $i= 1, \ldots , m$.
The topological boundary of $U$ consists of $m+1$ type II  points.

By the Poisson formula \cite[Proposition 7.23]{BR}, we may find real numbers $c_0, \ldots, c_m$ with $\sum_{i=1}^m c_i =0$ such that for all $z \in U$
$$g(z) = c_0 + \sum_{i=1}^m c_i \cdot \log |(T-a_i)(z)|.$$
Pick non-zero integers $n_1, \ldots, n_m$ such that $|c_i - n_i |< 1$ and $b\in k$ such that $| \log |b| - c_0| < 1 $.
Consider the map  $h: U \to \A^{1, \an} \setminus \{ 0 \}$,
$$h(z) = b \prod_{i=1}^m (T-a_i)^{n_i}(z).$$
Since $a_i \notin U$, the function $\log |h|$ is harmonic on $U$ and we have
\begin{equation*}
\sup_U |g - \log |h|| \le |c_0 - \log |b|| + \sum_{i=1}^m |c_i - n_i| \cdot \sup_U \log |(T-a_i)(z)|.
 \end{equation*}
 The functions $\log |(T-a_i)(z)|$  are bounded on $U$ and it follows that the right-hand side of the inequality is bounded.
\end{proof}

\subsection{Green functions after Kawaguchi-Silverman}

Consider an endomorphism of the $N$-dimensional  projective analytic space $f:\mathbb{P}^{N, \an} \to \mathbb{P}^{N, \an}$ of degree $d \ge 2$.
 Denote by $f^n$ its $n$-th iterate.
Fixing homogeneous coordinates, such a map can be written as $f= [F_0 : \cdots : F_N]$, with $F_i$ homogeneous polynomials of degree $d$ without non-trivial common zeros.

Denote by $\rho: \A^{N+1, \an}\setminus \{ 0\} \to \mathbb{P}^{N, \an}$ the natural projection map.
 An endomorphism $f$ of $\mathbb{P}^{N, \an}$ can be lifted to a map $F: \A^{N+1, \an} \to \A^{N+1, \an}$ such that $\rho \circ F = f \circ \rho$.
 One can take for instance $F= (F_0, \cdots, F_N)$.  In the sequel, we will always choose lifts of $f$ such that all the coefficients of the $F_i$'s lie in $k^\circ$ and at least one of them has norm 1. 
\medskip

Given $T_0, \ldots, T_N$ affine coordinates of $\A^{N+1, \an}$ and a point $z \in \A^{N+1, \an}$, we define its norm as $| z| = \max_{0 \le i \le N} |T_i(z)|$. Analogously, we set $|F (z)| = \max_{0 \le i \le N} |F_i(z)|$. 
With these norms in hand, we may now define the Green function associated to $f$ following Kawaguchi and Silverman \cite{Silvkawadynamics,SilverKawagreen}, see \cite{Sibony} for the complex case.

\begin{defprop}
The sequence of functions  
\begin{equation*}
G_n (z) =  \frac{1}{d^n} \log |F^n(z)|
\end{equation*}
converges uniformly on $\A^{N+1,\an}$.

One defines the dynamical Green function associated to $f$ as $G_f(z) = \lim_{n\to\infty} G_n$.
\end{defprop}

\begin{proof}
Let us show that the limit $\lim_n G_n$ exists.
The inequality $|F(z)| \le |z|^d$ is clear.
Since the polynomials $F_i$ have no common zeros other than the origin, by the homogeneous Nullstellensatz we may find a positive integer $s$ such that the homogeneous polynomial $T_i^s \in k[T_0, \ldots, T_N]$ belongs to the ideal generated by $F_0, \ldots , F_N$  for every $i=0, \ldots, N$. That is, for every $i$ there are homogeneous polynomials $\lambda_j^i  \in k[T_0, \ldots, T_N]$ such that $T_i^s = \sum_{j=0}^N \lambda_j^i F_j$. 
For any $z \in \A^{N+1,\an}$, we have:
\begin{equation*}
|z|^s = \max_{0\le i \le N}|z_i|^s \le \max_{0\le i, j \le N} | \lambda_j^i (z) F_j (z) | \le \max_{0\le i, j \le N}C |z|^{s-d} \cdot \max_{0\le j \le N} | F_j (z) | 
\end{equation*}
for some positive constant $C$ depending only on the polynomials $\lambda_j^i$.
Hence,  for all $z$ we have that 
\begin{equation}\label{eq null}
C \cdot |z|^d \le |F(z)| \le |z|^d,
\end{equation}
and so
\begin{equation*}
C \cdot |F^n(z)|^d \le |F^{n+1}(z)| \le |F^n(z)|^d.
\end{equation*}
 Set $ C_1= |\log C |$. Taking logarithms, one obtains
\begin{equation}\label{eq green cauchy}
\left| G_{n+1} - G_n \right| \le \frac{C_1}{d^n}.
\end{equation}
 By the ultrametric inequality, $\left| G_{n+j} - G_n \right| \le \frac{C_1}{d^n}$ for all $j \ge 0$ and for all $n$, and so the limit $G_f = \lim_{n \to \infty} G_n$ exists.
 \end{proof}
 
\begin{obs}
Letting $j$ go to infinity in \eqref{eq green cauchy}, one obtains the inequality
 \begin{equation}\label{eq bound G}
  \left| G_f - G_n \right| \le \frac{C_1}{d^n}.
 \end{equation}
 \end{obs}

\begin{theorem}[\cite{Silvkawadynamics}]\label{thm properties green}
\begin{enumerate}\renewcommand{\labelenumi}{\roman{enumi})}
\item The function $G_f$ is continuous.
\item For every $\lambda \in k^*$ and for every $z \in \A^{N+1, \an}$, we have that
$$G_f(\lambda \cdot z) = G_f(z) + \log | \lambda|.$$
\item There exists a positive constant $C$ such that  
$$\sup_{z \in \A^{N+1,\an}} |G_f(z) - \log |z| | \le C~.$$
\end{enumerate}
\end{theorem}

\subsection{Fatou and Julia sets}

Let us first discuss the one-dimensional situation, both in the complex and in the non-Archimedean setting.

Recall that there are several characterizations of the  Fatou and Julia  sets of an endomorphism $f$ of $\mathbb{P}^1_\C$.
The Fatou  set $F(f)$ can be defined as the normality locus of the family of the iterates of $f$, and  the Julia set $J(f)$ as its complement.
Equivalently,  one can set $J(f)$  to be the support of unique measure of maximal entropy, also referred to as the equilibrium   measure, see \cite{Sibony}, or as  the closure of the repelling periodic  points.

Some of these equivalences  have a non-Archimedean counterpart.
There is a well-defined notion of the canonical measure of an endomorphism $f$ of $\mathbb{P}^{1,\an}$ (see \cite{FavreRLbrolin,FavreRLeq} and \cite[\S 10.1]{BR}),
 and so  one sets $J(f)$ to be its support and $F(f)$ its complement.
Using a similar definition of normality as ours, it can be shown that the Fatou set agrees with the normality locus of the family of the iterates of $f$ \cite[Theorem 5.4]{FKT}.

\medskip

One may as well consider the Fatou and Julia sets in restriction to the set of rigid points of $\mathbb{P}^{1,\an}$, see \cite{Silvermanbook} for a survey on the topic.
 However, notice that if $f$ is a map with good reduction, i.e. if the reduction $\tilde{f}$ of $f$ is a selfmap of $\mathbb{P}^1_{\tilde{k}}$  of the same degree as $f$, then its Julia set contains no rigid points  \cite[Theorem 2.17]{Silvermanbook}. %reason: equicontinuity

We mention  the following two characterizations of the intersections of $J(f)$ and $F(f)$ with $\mathbb{P}^{1, \an}(k)$.
It was shown in \cite[Theorem C]{FKT} that  the intersection of the Fatou set $F(f)$ with the set of rigid points in $\mathbb{P}^{1, \an}$  agrees with the set of rigid points where the sequence of the iterates $f^n$ is equicontinuous with respect to the chordal metric on $\mathbb{P}^{1,\an}(k)$.

\medskip

The Fatou set of a non-invertible complex endomorphism $f$ of $\mathbb{P}^N_\C$ for $N\ge 2$ is defined as the normality locus of the family of the iterates. Its complement is the support of the Green current, which is the unique positive closed $(1,1)$-current that is forward invariant by $f$,  see \cite[Th\'eor\`eme 1.6.5]{Sibony} for 
a proof. There are several possible definitions for the Julia set of $f$, see \cite[D\'efinition 3.31]{Sibony}. We define the Julia set of $f$ as the complement of the Fatou set.

\medskip

We now explore the non-Archimedean higher dimensional case.
We consider two different Fatou sets of $f$:

\begin{defini}
The \emph{normal Fatou set} $F_{\mathrm{norm}} (f)$ of an endomorphism $f: \mathbb{P}^{N,\an} \to \mathbb{P}^{N,\an}$ of degree at least 2 is 
the set of all points $ z \in \mathbb{P}^{N, \an}$ where the family $ \{ f^n \} $ is normal.

The normal Julia set $J_{\mathrm{norm}}  (f)$ is the complement of $F_{\mathrm{norm}} (f)$.
\end{defini}

\begin{defini}
Let $\rho : \A^{N+1, \an} \to \mathbb{P}^{N,\an}$ be the usual map.
We define the harmonic Fatou set $F_{\mathrm{harm}} (f)$ of $f$ as the set of points $z \in \mathbb{P}^{N,\an}$ having a neighbourhood $U$ such that the Green function $G_f$ is strongly pluriharmonic on $\rho^{-1} (U)$.

The harmonic Julia set $J_{\mathrm{harm}}  (f)$ is the complement of $F_{\mathrm{harm}} (f)$.
\end{defini}

It follows directly from the definitions that both Fatou sets $F_{\mathrm{norm}} (f)$ and  $ F_{\mathrm{harm}} (f)$ are open  and  totally invariant.

The set $J_{\mathrm{harm}} (f)$ is always nonempty. Indeed,
Chambert-Loir has constructed a natural invariant probability measure $\mu_f$ on $\mathbb{P}^{N,\an}$ and shown that its support is
 contained in the complement of the locus where $G_f$ is strongly pluriharmonic, see \cite[Proposition 2.4.4]{CLheights}. 
In other words, the support of $\mu_f$ is included in the harmonic Julia set of $f$.

We do not know whether the Fatou set is always non-empty.

\begin{example}
Let $z \in \mathbb{P}^{N,\an}$ be any rigid fixed point for $f$ such that the eigenvalues of its differential $Df(z)$ are all of norm at most $1$.
Then, we may find an arbitrarily small open neighbourhood $U$ of $z$ which is $f$-invariant, i.e. such that $f(U) \subseteq U$.
After maybe reducing $U$, we may assume 
that  $U \subset \{ z_0 = 1, |z_i| <2, i=1, \cdots , N\}$. We thus have:
\begin{eqnarray*}
G_n &=& \frac{1}{d^n} \log \left|(F_0^n, \cdots , F_N^n)  \right| \\
& = &\frac{1}{d^n} \log \left| F_0^n \right|   + \frac{1}{d^n} \log \max_{1 \le i \le N} \left|\frac{F_i^n}{F_0^n} \right|~. 
\end{eqnarray*}
The second term converges uniformly to $0$.
 On the open set $\rho^{-1}(U)$,  the function $G_f$ is thus the uniform limit of the sequence $\frac{1}{d^n} \log \left| F_0^n \right|$, hence strongly pluriharmonic.
Hence $z$ belongs to the harmonic Fatou set.
\end{example}

In dimension $1$, it follows from the Woods Hole formula that 
any rational map admits at least one indifferent fixed point $p$, i.e. such that $|f^\prime (p)| = 1$.
We observe that the same result holds for any polynomial map  $f: \A^{2,\an} \to \A^{2,\an}$ that extends to an endomorphism of $\mathbb{P}^{2,\an}$
so that  $F_{\mathrm{harm}} (f)\neq \emptyset$ in this case.

\medskip

\begin{obs}
In \cite{SilverKawagreen}, the authors define the Fatou set of an endomophism  of the $N$-th projective space $\mathbb{P}^N_k$ as the equicontinuity locus of the family of iterates, which they prove to be the same as the locus where it is locally uniformly Lipschitz. However, the definition of the Fatou set in terms of equicontinuity presents some difficulties already in dimension one.
Indeed, let  $k$ be a field of characteristic $p >0$ and consider the polynomial $f(z) = pz^2+cz$, with $|c| = 1$. 
Then, the family of the iterates $f^n$ is normal at the Gauss point, but it is not equicontinuous at $x_g$,
 see \cite[Example 10.53]{BR}.
\end{obs}

\subsection{Comparison between $F_{\mathrm{norm}}$ and $F_{\mathrm{harm}}$}

We expect our two notions of Fatou sets to coincide. 
\begin{conj}
For every non-invertible endomorphism $f$ of the projective space, we have that $F_{\mathrm{norm}} (f) = F_{\mathrm{harm}} (f)$.
\end{conj}

In dimension $1$, the equality follows from
\cite[Theorem 5.4]{FKT}, and we are able to prove one inclusion in general. Our argument relies on the
 following result which gives a characterization of $F_{\mathrm{harm}} (f)$ in terms of a sort of equicontinuity property for the iterates of $f$.
 Its proof follows its complex  counterpart.

\begin{prop}\label{prop fatou pluriharm}
Let $f: \mathbb{P}^{N, \an} \to \mathbb{P}^{N, \an}$ be an endomorphism of degree $d\ge 2$ and $U$  a basic tube  in $\mathbb{P}^{N,\an}$.

The Green function $G_f$ is strongly pluriharmonic on the open set $\rho^{-1}(U) \subset \A^{N+1, \an}  \setminus \{ 0 \}$
if and only if for every $n \in \N$ 
 there exists a lift $F_n$ of $f^n$ on $U$ and a positive constant $C_1$ such that $e^{-C_1} \le | F_n | \le e^{C_1}$ on $\rho^{-1} (U)$ for all $n \in \N$.
\end{prop}

This result together with
 Theorem \ref{thm a} implies the following:
\begin{cor}
The harmonic Fatou set $F_{\mathrm{harm}} (f)$ is contained in $F_{\mathrm{normal}} (f)$. 
\end{cor}

\begin{proof}[Proof of Proposition \ref{prop fatou pluriharm}]
Pick any lift $F= (F_0, \cdots, F_N)$ of $f$,  where $F_i \in k[T_0, \cdots , T_N]$ are homogeneous polynomials of degree $d$ without nontrivial common zeros.  We may assume that $\sup_{\D} | F(z)| = 1$. 
Recall from \eqref{eq bound G} that there exists a positive constant $C_1$ such that  $\left|G_f-G_n \right| \le \frac{C_1}{d^n}$ for all $n \in \N$. 

Let $U$ be a basic tube on which $G_f$ is strongly pluriharmonic. Let $h_n \in \mathcal{O}_{\A^{N+1}}^\times(U)$ and let $b_n$ be non-zero real numbers   such that  $G_f$ is the uniform limit of the sequence $b_n \cdot \log |h_n|$.
After maybe extracting a subsequence and renumbering it, we may assume that 
\begin{equation*}
\left| G_f - b_n \cdot \log |h_n| \right| \le \frac{C_1}{d^n} \quad \forall n\gg 0
\end{equation*}
on U. Thus, we have
\begin{eqnarray*}
\left| \frac{1}{d^n} \log | F^n| -  b_n \cdot \log |h_n|  \right| & = & \left| \frac{1}{d^n }\log \left( \frac{| F^n| }{|h_n|^{b_n \cdot d^n}} \right) \right| \\
& \le & \max \left\{ \left|G_f -  b_n \cdot \log |h_n| \right| ,  \left| G_f-G_n \right| \right\} \\
& \le &  \frac{C_1}{d^n}~.
\end{eqnarray*}

So we see that for $n \gg 0$
\begin{equation}\label{bound lifts}
e^{-C_1} \le \frac{| F^n| }{|h_n|^{b_n \cdot d^n}} \le e^{C_1}.
\end{equation}
Since the functions $h_n$ have no zeros on $U$, each $F_n : = \frac{F^n}{h_n^{b_n \cdot d^n}}$ is a lift of $f^n$. 
\medskip

Assume conversely that on $U$, for every $n \in \N$  there exists a  lift $F_n$ of $f^n$ such that  $e^{-C_1} \le |F_n| \le e^{C_1}$  for some  positive constant $C_1$.
Then, for every $n \in \N$ we may choose a non-vanishing function $h_n$ on $U$ such that $F^n = h_n \cdot F_n$.
It follows that 
\begin{equation*}
G_n = \frac{1}{d^n} \log \left| F^n  \right| 
 = \frac{1}{d^n} \log \left| h_n \right|   + \frac{1}{d^n} \log \left| F_n  \right| ~.
\end{equation*}
The second term converges uniformly to $0$.
 On the open set $\rho^{-1}(U)$,  the function $G_f$ is thus the uniform limit of the sequence $\frac{1}{d^n} \log \left| h_n \right|$, hence strongly pluriharmonic.
\end{proof}

\subsection{Hyperbolicity of the Fatou components}

Recall that $\mathrm{Mor}_k(X,Y)$ denotes the set of analytic maps from $X$ to $Y$.

\begin{defini}
Let $\Omega$ be a relatively compact subset of an analytic space $Y$ and   $U$ a basic tube.

The family $\mathrm{Mor}_k(U, \Omega)$ is said to be \emph{normal} if for every sequence of analytic maps $\{ f_n \} \subset \mathrm{Mor}_k(U, \Omega)$ there exists a subsequence $f_{n_j}$ that is pointwise converging to a continuous map $f: U \to Y$. 
\end{defini}

\begin{obs}
In the complex setting, the previous definition corresponds to the 
 family $\mathrm{Hol}(U,\Omega)$ being relatively compact in $\mathrm{Hol}(U,Y)$.
The complex definition of normality for a non-compact target is slightly different, since it allows for a sequence to be compactly divergent \cite[\S I.3]{Kobbook}.
\end{obs}

%XXXX Thesis: discussion of tautness.

Let  $f:\mathbb{P}^{N, \an} \to \mathbb{P}^{N, \an}$ be an endomorphism of degree at least 2.
Theorem \ref{thm c} thus states that for every connected component $\Omega$ of the harmonic Fatou set $F_{\mathrm{harm}} (f)$  and for every  connected open subset $U$ of $\mathbb{P}^{1, \an}$, 
the family $\mathrm{Mor}_k(U, \Omega)$ is normal.

\begin{proof}[Proof of Theorem \ref{thm c}]
Let $\Omega$ be a connected component of  $F_{\mathrm{harm}} (f)$ of an endomorphism $f:\mathbb{P}^{N, \an} \to \mathbb{P}^{N, \an}$ of degree at least 2. 
 Let $U$ be any connected open subset of $\mathbb{P}^{1, \an}$. 
Our aim is to show that the family $\mor_k (U, \Omega)$ is normal.

\smallskip
The projective space $\mathbb{P}^{N, \an}$ can be covered by $N+1$ charts $V_0, \ldots, V_N$ analytically isomorphic  to $\bar{\D}^N$.
For every $i=0, \cdots, N$, let $s_i: \{z \in \mathbb{P}^ {N, \an} : z_i \neq 0 \} \to \A^{N+1, \an}$ be the analytic local section of $\rho$ sending the point $z=[z_0: \ldots : z_N]$ to $(\frac{z_0}{z_i}, \ldots , \frac{z_{i-1}}{z_i}, 1, \frac{z_{i+1}}{z_i}, \ldots, \frac{z_N}{z_i} ) $. 
Let  $g: U \to \Omega$ be an analytic map. 
We claim that for any compact subset $\mathrm{K} \subset U$ the map $g_{|_\mathrm{K}}$ admits a lift to $\rho^{-1}(\Omega)$.

\smallskip

Suppose first that $U$ is not the whole $\mathbb{P}^{1, \an}$. By Proposition \ref{prop tubes BR}, there exists a sequence of basic tubes  $W_m$ exhausting $U$  and a sequence of affinoid subspaces $X_m$  satisfying 
$$\overline{W}_m \subset X_m \subset U~.$$

Pick any compact subset $\mathrm{K} \subset U$. For $m$ sufficiently large, $\mathrm{K}$ is contained in some $X_m$.
Fix $m \in \N^*$. Cover $X_m$ by sets $U^{(m)}_i = g^{-1}(V_i) \cap X_m$ with $0 \le i \le N$.
 On every  $U^{(m)}_{ij} = g^{-1}(V_i) \cap g^{-1}(V_j) \cap X_m $, 
 we know that $\rho \circ s_i \circ g = \rho \circ s_j \circ g$, and thus $s_i \circ g =\varphi^{(m)}_{ij} \cdot ( s_j \circ g)$ for some $\varphi^{(m)}_{ij} \in \mathcal{O}^\times(U^{(m)}_{ij})$.
Since $X_m$ is an affinoid subspace of $\mathbb{P}^{1, \an}$ we have that $H^1(X_m, \mathcal{O}^\times ) = 0$ by   \cite{Put1980}.
 We may thus find  $\varphi_i \in \mathcal{O}^\times(U^{(m)}_i)$ and $ \varphi_j \in \mathcal{O}^\times(U^{(m)}_j)$
 such that $\varphi^{(m)}_{ij} = \frac{\varphi^{(m)}_i}{ \varphi^{(m)}_j}$. 
On   $X_m$, consider the following local lifts of $g$:
\begin{equation*}
\widehat{g_i}^m : U^{(m)}_i \to \rho^{-1}(\Omega), \quad \widehat{g_i}^m= \frac{s_i \circ g}{\varphi^{(m)}_i}.
\end{equation*}
It follows that $\widehat{g_i}^m= \widehat{g_j}^m$ on $U^{(m)}_{ij}$, and hence we have a lift $\widehat{g}^m: X_m \to \rho^{-1} (\Omega)$ of $g$ as required.

\medskip

By definition of the harmonic Fatou set,  the Green function $G_f$ of $f$ is strongly pluriharmonic on $\rho^{-1}(\Omega)$, and thus  $G_f \circ\widehat{g}^m$ is harmonic on $X_m$.

Let $g_n: U \to \Omega$ be a sequence of analytic maps.
For every $X_m$ consider the lifts $\widehat{g_n}^m: X_m \to \rho^{-1}(\Omega)$ of the restriction of $g_n$ to $X_m$ constructed  above.

Fix a sufficiently large real number $C >0$ and consider the set  $\mathrm{M}= \{z \in \A^{N+1, \an}  \setminus \{ 0 \} : \frac{1}{C} \le |G_f(z)| \le C \}$. 
By Theorem \ref{thm properties green}, the set $\mathrm{M}$ is compact. By Proposition \ref{lema bound harmonic functions}, for every $n$ and every $m$ there exists an analytic map $h_n^m: W_m \to \A^{1, \an} \setminus \{ 0 \}$ such that
$$\sup_{W_m} \big| G_f \circ \widehat{g_n}^m - \log |h_n^m| \big| \le C.$$
We set $\widetilde{g_n}^m = \frac{\widehat{g_n}^m}{h_n^m}$.
 Each $\widetilde{g_n}^m: W_m \to \rho^{-1}(\Omega)$ is a lift of $g_n$  and its image lies in the compact $\mathrm{M}$.
 By Theorem \ref{thm a}, there exists a subsequence of $\widetilde{g_n}^m$ converging pointwise to a continuous map.
 By a diagonal extraction argument, we conclude that the family $\mathrm{Mor}_k (U, \Omega)$ is normal.
 \medskip

The case $U=\mathbb{P}^{1, \an}$ follows by writing $\mathbb{P}^{1,\an}$ as a finite union open disks. 
\end{proof}

\subsection{Curves in Fatou sets}
The aim of this section is to prove Theorem \ref{thm d}, i.e. to show that harmonic Fatou components 
contain no nontrivial image of $\A^{1,\an} \setminus \{ 0\}$.

\medskip

We briefly observe the following fact that  follows almost directly from the work of Chambert-Loir.

\begin{prop}\label{prop algebraic curve}
Suppose that $C$ is an algebraic curve in $\mathbb{P}^{N, \an}$, and let  $f:\mathbb{P}^{N, \an} \to \mathbb{P}^{N, \an}$ be any endomorphism of degree at least 2. 
Then the harmonic  Fatou  set of $f$ cannot contain a Zariski open subset of $C$.
\end{prop}

In particular, a Fatou component contains no complete algebraic curve. This supports the conjectural   fact that any Fatou component should be Stein (in the sense of \cite{Kiehl}). 
Over the complex numbers, this result is proved in~\cite{Ueda,FS2,Maegawa}, but the proof relies on pluripotential techniques which are not available at the moment over a non-Archimedean field.

\begin{proof}[Proof of Proposition \ref{prop algebraic curve}]
Since the result is not central to our studies, we shall only give a sketch of proof,
which relies on special metrizations of line bundles. We refer to~\cite[\S 2]{CLheights} for a detailed exposition of these notions. 
Choose a homogeneous lift $F= (F_0, \cdots, F_N)$ of $f$ to $\A^{N+1,\an}_k\setminus\{0\}$, and consider the associated Green function
$G_f = \lim_n \frac1{d^n} \log|F^n|$. The function $G_f$ induces a continuous and semi-positive metrization $|\cdot|_F$ in the sense of Zhang on the 
 tautological line-bundle $\mathcal{O}(1)$ on $\mathbb{P}^{N,\an}$, see  \cite[\S 2.1]{CLheights}.

Pick any algebraic curve $C$ in $\mathbb{P}^{N, \an}$. The restriction of the metrized line bundle $(\mathcal{O}(1), |\cdot|_F)$
to $C$ is again continuous and semi-positive. We may thus consider its curvature, see~\cite[Proposition 4.2.3]{Thuillier}. 
It is a positive measure $\mu_C$  on the Berkovich analytification of $C$ of mass $\deg_C (\mathcal{O}(1))$
which does not charge any rigid point, see \cite[\S 4.2.1]{Thuillier}.
The support of $\mu_C$ is contained in $J_{\mathrm{harm}}(f)$, which
 implies the result.
\end{proof}

We shall use the following  proposition:

\begin{prop}\label{prop punctured brody}
Let $\Omega$ be an open subset of $\mathbb{P}^{N,\an}$.

 If  the family of analytic maps $\mathrm{Mor}_k(\A^{1, \an} \setminus \{ 0 \}, \Omega)$ is normal, then every analytic map $\A^{1, \an} \setminus \{ 0 \} \to \Omega$  is constant.
\end{prop}

As a direct application, we obtain:

\begin{proof}[Proof of Theorem \ref{thm d}]
It follows from Theorem \ref{thm c} and Proposition  \ref{prop punctured brody}.
\end{proof}

As a first step in proving Proposition \ref{prop punctured brody}, we deal with a simpler particular case, that of entire curves.

\begin{proof}[Proof of the particular case of entire curves]
Let $\Omega$ be any open subset of $\mathbb{P}^{N,\an}$ and assume  that the family  $\mathrm{Mor}_k(\A^{1, \an}, \Omega)$ is normal.
Suppose that
 there exists a non-constant analytic map $g: \A^{1, \an} \to \Omega$.
 Consider the sequence of analytic maps from $\A^{1, \an}$ into $\Omega$ given by   $f_n(z) = g (z^n)$. 
 By normality  there is a subsequence $\{ f_{n_j} \}$  that is pointwise converging to a continuous map $f : \A^{1,\an} \to \mathbb{P}^{N,\an}$.

The Gauss point $x_g$ is fixed by all the maps $z\mapsto z^n$, and so $f(x_g) = g(x_g)$. 
For every integer $m >0$ let $z_m = \eta_{0, 1-\frac{1}{m}} \in \A^{1, \an}$.
 Since every $z_m$ lies in the open unit disk $\D$, we have that
\begin{equation*}
f(z_m) = \lim_{n_j \to \infty} f_{n_j} ( z_m) = \lim_{n_j \to \infty} g\left( (z_m)^{n_j} \right) = g(0)
\end{equation*}
 for all $m$. The continuity of $f$ implies that the $f(z_m)$ tend to $f(x_g)$ as $m$ goes to infinity. It follows that $g(x_g) = g(0)$ is a rigid point of $\Omega$. As the source $\A^{1,\an}$ is one-dimensional, $g$  must be constant.
\end{proof}

In order to prove Proposition \ref{prop punctured brody}, we need to recall some basic topological facts. 
Recall from \S \ref{section harmonic} that given a point $x \in \mathbb{P}^{1,\an}$, we denote by $U(\vec{v})$ the connected component of $\mathbb{P}^{1,\an} \setminus \{ x\}$ corresponding to the tangent
 direction $\vec{v} \in T_x \mathbb{P}^{1,\an}$.

Let $g : U \subseteq  \mathbb{P}^{1,\an} \to \mathbb{P}^{1,\an}$ be a non-constant  analytic  map.
For every  point $x\in U$, the map $g$ induces a tangent map $dg(x) $ between $T_x U $ and $T_{g(x)} \mathbb{P}^{1, \an}$. 
Let $\vec{v}$ be a tangent direction at $x$ that is mapped to $\vec{v}' \in T_{g(x)} \mathbb{P}^{1,\an}$ by $dg(x)$. Then either $g(U(\vec{v})) = U(\vec{v}')$ or $g(U(\vec{v})) = \mathbb{P}^{1,\an}$. This follows from the fact that the map $g$ is open \cite[Corollary 9.10]{BR}.

Of special interest for us  is the case when $x$ is a type II point. Assume for simplicity that both $x$ and $g(x)$ are the Gauss point.
The space $T_{x_g} \mathbb{P}^{1,\an}$ is  isomorphic to  $\mathbb{P}^1_{\tilde{k}}$,
and the tangent map $dg (x): \mathbb{P}^1_{\tilde{k}} \to \mathbb{P}^1_{\tilde{k}}$  and   can be described as follows.
In homogeneous coordinates $g$ can be written as $g=[G_0: G_1]$ with $G_0, G_1 \in \mathcal{O}(\A^{1,\an})$ without common zeros by \cite[Theorem 2.7.6]{vanderputFresnel}, 
where all the coefficients of $G_0$ and $G_1$ are of norm less or equal than one and  least one has norm one. Thus, we may consider the reduction map of $g$, which  is a  non-constant rational map from $\mathbb{P}^1_{\tilde{k}}$ to itself, and hence surjective.
One can show that $dg(x)$ is given by the reduction of $g$ \cite[Corollary 9.25]{BR}.

\begin{proof}[Proof of Proposition \ref{prop punctured brody}]
Suppose that  $\mathrm{Mor}_k(\A^{1, \an}\setminus \{ 0 \}, \Omega)$ is normal.
We first deal with the case where  $\Omega$ is contained in $\mathbb{P}^{1, \an}$.
 Let $g: \A^{1,\an} \setminus \{ 0 \} \to \mathbb{P}^{1,\an}$ be a  non-constant analytic map.
We may assume that it is of the form $g=[G_0: G_1]$ with $G_i : \A^{1, \an} \setminus \{ 0 \} \to \A^{1, \an}$ analytic without common zeros by \cite[Theorem 2.7.6]{vanderputFresnel}.
Our goal is to construct a sequence of analytic maps from $\A^{1, \an}\setminus \{ 0 \}$ to itself such that the composition with $g$ gives a sequence $g_n: \A^{1, \an}\setminus \{ 0 \} \to \Omega$ that admits no converging subsequence with continuous limit.

\medskip

Suppose first that there exists a type II point in $\mathbb{P}^{1, \an}$ 
having infinitely many preimages in the segment $T = \{\eta_{0,r} \in \A^{1,\an} : 0<r<\infty \}$. 
 Composing with an automorphism of $\mathbb{P}^{1,\an}$, we may assume that this point is the Gauss point.
Let thus $\{ \eta_{0,r_n} \}$  be a sequence of preimages of $x_g$.

 Denote by $V_n$ the compact set containing $\eta_{0 , r_n} $ consisting of $\A^{1,\an} \setminus \{ 0 \}$ minus the open sets $U(\vec{v}_0)$ and  $U(\vec{v}_\infty)$, where $\vec{v}_0$ and $\vec{v}_\infty$ are the tangent directions at $\eta_{0, r_n} $ pointing at 0 and at infinity respectively.
 As $dg (\eta_{0, r_n})$ is surjective, we deduce that $g(V_n)$  avoids at most two tangent directions at $x_g$.
 After maybe extracting a subsequence, we may find a connected component $B$ of $\mathbb{P}^{1, \an} \setminus \{ x_g \}$   that is contained in  $g(V_n)$ for all $n \gg 0$.
As a consequence, we may pick a rigid point $a_0$ in $B$ and rigid points $x_n \in V_n$ such that $g(x_n) = a_0$ for every $n \in \N$.

Consider  the sequence  in $\mathrm{Mor}_k(\A^{1,\an} \setminus \{ 0 \}, \mathbb{P}^{1, \an})$ given by  $g_n (z)= g(x_{n!}  z^{n!})$. 
By normality, we may assume that $g_n$  converges  to a continuous map $g_\infty$.
The Gauss point $x_g$ is fixed by  $g_\infty$, as $g_n(x_g) = x_g$ for all $n \in \N$. 
For every fixed $n \in \N$ and  every $m\le n$, the map $g_n$ sends the set of all the $m$-th roots of unity  $R_m$ to $a_0$,
and so $g_\infty$  maps every $R_m$ to $a_0$.
For every $m \in \N$ pick a point $\zeta_m \in R_m$ such that $\zeta_m \to x_g$ as $m$ tends to infinity. We have that
 $$g_\infty (x_g) = \lim_{m \to \infty} g_\infty (\zeta_m) =  a_0~,$$
 contradicting  the continuity.
\medskip

Suppose next that every type II point in $\mathbb{P}^{1, \an}$ has at most finitely many preimages in the segment $T$.
 Pick a sequence of type II points $\{ \eta_{0 , r_n} \}$ with  $r_n \to + \infty$ as $n$ goes to infinity. 
By compactness, we may assume that  the points $g( \eta_{0,r_n})$ converge to some point $y_\infty \in \mathbb{P}^{1, \an}$.
We claim that the points $g( \eta_{ 0 , r})$ converge to a point $y_\infty$ as $r$ tends to infinity. To see this,
fix a  basic tube $V$ containing $y_\infty$. Recall  that $\partial_{\mathrm{top}} V$ is a finite set of type II points.  
By assumption,  $g( \eta_{ 0 , r} )$ does not belong to  $\partial_{\mathrm{top}} V$  for sufficiently large $r$.
For $n\gg 0$ we have that $g(\eta_{ 0 , r_n})$ lies in $V$. Thus, $g(\eta_{ 0 , r})$ must belong to $V$ for $r \gg 0$.

Pick any $r\in \R_+$ and consider the tangent direction $\vec{v}$ at $\eta_{0,r}$ pointing towards infinity.  We may assume that   $g(U(\vec{v}))$ avoids at most one rigid point in $\mathbb{P}^{1, \an}$, as otherwise Picard's Big theorem \cite{Cherrypicard} asserts that $g$  admits an analytic extension at  infinity  and we conclude by the case of entire curves.
After maybe varying the $r_n$,  we may find a rigid point $a_0 \in \mathbb{P}^{1,\an}$ and rigid points
 $x_n$   with $|x_n| = r_n$ such that $g(x_n) = a_0$ for all $n$.

Consider the sequence $g_n(z) = g(x_{n!} z^{n!})$ and assume that it admits a  continuous limit $g_\infty$.
Our previous argument shows that $g_\infty$ maps every set $R_m$ to $a_0$.
 The points $g_n(x_g)$ converge to $y_\infty$ by our claim, and hence $g_\infty$ is not continuous.

\medskip

Assume now that $\Omega$ is an open subset of $\mathbb{P}^{N, \an}$. Let $g: \A^{1, \an} \setminus \{ 0 \} \to \Omega$ be a non-constant analytic map.
This map can be written in homogeneous coordinates as $g= [G_0 : \ldots : G_N]$, with $G_i \in \mathcal{O}^\times (\A^{1, \an} \setminus\{ 0 \})$.
As  $g$ is not constant we may assume that $G_0$ is non-constant and that $G_1$ is not a scalar multiple of $G_0$.
We may assume by \cite[Theorem 2.7.6]{vanderputFresnel} that $G_0$ and $G_1$ have no common zeros.
 As a consequence, the map defined on the image of $g$ by
 $$\pi: [G_0(z) : \ldots : G_N(z)] \mapsto [ G_0(z): G_1(z)]$$
is well-defined and analytic.
By construction $\pi \circ g$ is non-constant and analytic.
By the previous case we may find $x_n \in k^\times$ such that no subsequence of $\{ \pi \circ g (x_{n!} z^{n!} )\}$ has a continuous limit, and thus neither $\{ g (x_{n!} z^{n!} )\}$.
\end{proof}

\bibliographystyle{amsalpha}

\bibliography{biblio}

\end{document}